\documentclass[12pt]{amsart}
\usepackage{amssymb,amsmath,amsthm}
\usepackage{float,graphicx}
\usepackage[linktocpage=true,backref=page]{hyperref}
\hypersetup{
    colorlinks=true, %set true if you want colored links
    linktoc=all,     %set to all if you want both sections and subsections linked
    linkcolor=magenta,  %choose some color if you want links to stand out
    citecolor=cyan,
}
\restylefloat{table}
\newcommand*{\Scale}[2][4]{\scalebox{#1}{$#2$}}%
\calclayout
\newtheorem{theorem}{Theorem}[section]
\newtheorem{proposition}[theorem]{Proposition}
\newtheorem{corollary}[theorem]{Corollary}
\newtheorem{lemma}[theorem]{Lemma}
\newtheorem{remark}[theorem]{Remark}
\numberwithin{theorem}{section} \numberwithin{equation}{section}
\newcommand\mycom[2]{\genfrac{}{}{0pt}{}{#1}{#2}}
\newcommand{\hpg}[5]{{}_{#1}F_{#2}\! \left(\left.{\genfrac{}{}{0pt}{}{#3}{#4}}\right| #5 \right) }
\newcommand{\hpgmod}[5]{{}_{#1}F^{(\epsilon)}_{#2}\! \left(\left.{\genfrac{}{}{0pt}{}{#3}{#4}}\right| #5 \right) }

\newcommand{\hpgo}[2]{{}_{#1}F_{#2}}
\newcommand{\app}[4]{F_{#1}\! \left(\left.{\mycom{#2}{#3}}\right| #4 \right) }
\newcommand{\beq}{\begin{equation}}
\newcommand{\eeq}{\end{equation}}
\newcommand{\beqn}{\begin{equation*}}
\newcommand{\eeqn}{\end{equation*}}
\title[On the mixed-twist construction and Picard-Fuchs systems]{On the mixed-twist construction and monodromy of associated Picard-Fuchs systems}
\author{Andreas Malmendier}
\address{Dept.\!~of Mathematics, University of Connecticut, Storrs, Connecticut 06269\\
\hspace*{0.3cm} Dept.\!~of Mathematics \& Statistics, Utah State University, Logan, UT 84322}
\email{andreas.malmendier@uconn.edu}
\thanks{A.M. acknowledges support from the Simons Foundation through grant no.~202367.}
\author{Michael T.  Schultz}
\address{Dept.\!~of Mathematics, Virginia Tech, Blacksburg, VA 24060}
\email{michaelschultz@vt.edu}
\begin{document}
\begin{abstract}
We use the mixed-twist construction of Doran and Malmendier to obtain a multi-parameter family of K3 surfaces of Picard rank $\rho \ge 16$. Upon identifying a particular Jacobian elliptic fibration on its general member, we determine the  lattice polarization and the Picard-Fuchs system for the family. We construct a sequence of restrictions that lead to extensions of the polarization by two-elementary lattices. We show that the Picard-Fuchs operators for the restricted families coincide with known resonant hypergeometric systems. Second, for the one-parameter mirror families of deformed Fermat hypersurfaces we show that the mixed-twist construction produces a non-resonant GKZ system for which a basis of solutions in the form of absolutely convergent Mellin-Barnes integrals exists whose monodromy we compute explicitly. 
\end{abstract}
\maketitle
\bibliographystyle{abbrv}
\section{Introduction}
In \cite{MR4069107}, Doran and Malmendier introduced the \emph{mixed-twist construction}, which iteratively constructs families of Jacobian elliptic Calabi-Yau $n$-folds $Y^{(n)}$ from a family of Jacobian elliptic Calabi-Yau $(n-1)$-folds $Y^{(n-1)}$ for all $n \geq 2$. In fact, the new families are then fibered by the Calabi-Yau $(n-1)$-folds $Y^{(n-1)}$ in addition to being elliptically fibered. For example, for $n=2$ the procedure starts with a family of elliptic curves with rational total space, and the mixed-twist construction returns families of Jacobian elliptic K3 surfaces polarized by the lattice $H \oplus E_8(-1) \oplus E_8(-1) \oplus \langle -2k\rangle$ for certain $k \in \mathbb{N}$. The central tool of the construction, which is inspired by aspects of physics related to mirror symmetry and the embedding of F theory into gauge theory, is an invariant for ramified covering maps $\mathbb{P}^1 \to \mathbb{P}^1$, called the \emph{generalized functional invariant}. 
\par Central to the mixed-twist construction is the incarnation of an iterative relation between the period integrals of $n$-folds $Y^{(n)}$ and the periods of $Y^{(n-1)}$. When applied to the the family of mirror manifolds $Y^{(n-1)}$ of the family of deformed Fermat hypersurfaces  $X^{(n-1)}$ in $\mathbb{P}^n$
\begin{equation*}
X_0^{n+1} + \cdots + X_n^{n+1} + n\lambda X_0 X_1 \cdots X_n = 0 
\end{equation*}
obtained by the Greene-Plesser orbifolding construction~\cite{MR1137063}, Doran and Malmendier proved the existence of transcendental cycles ~$\Sigma_{n-1} \in H_{n-1}(Y^{(n-1)},\mathbb{Q})$ such that the period integral 
\begin{equation*}
    \omega_{n-1} = \int_{\Sigma_{n-1}} \, \eta^{(n-1)}
\end{equation*}
can be computed iteratively from the Hadamard product of the hypergeometric function $\hpgo{n}{n-1}$ and the period integral $\omega_{n-2}$ on $Y^{(n-2)}$  \cite[Prop. 7.2]{MR4069107}. Here, $\eta^{(n-1)}$ is a holomorphic trivializing section of the canonical bundle $K_{Y^{(n-1)}}$. We recall this result explicitly in Proposition \ref{PeriodLemma}. This result matches well known results in the literature on the periods of the mirror family $Y^{(n-1)}$, but elucidates the connection between the periods and the iterative fibration structure. 
\par In such a situation, of particular interest are the Picard-Fuchs operators that annihilate the periods $\omega_{n-1}$, and the monodromy behavior of the periods as one encircles singular points in family of Calabi-Yau varieties $Y^{(n-1)}$. In the context of mirror symmetry, the Picard-Fuchs operators are often realized as resonant GKZ hypergeometric systems \cite{MR1319280,MR1672077}  -- named after the seminal work by Gel'fand, Kapranov, and Zelevinsky \cite{MR1080980} -- a vast generalization of the hypergeometric function $\hpgo{n}{n-1}$. Due to resonance of these systems, the monodromy representations are reducible due to a result of Schulze and Walther \cite{MR2966708}, which makes their explicit determination much more challenging in general. In the case described above, the monodromy group of the hypergeometric Fuchsian ODE annihilating $\hpgo{n}{n-1}$ is known, going back to work of Levelt \cite{MR0145108}. The mixed-twist construction offers an alternative formulation to arrive at the same monodromy group (up to conjugacy) based off the iterative period relation. 
\par This article aims to demonstrate that the mixed twist construction is a suitable tool that allows for the computation of the monodromy group of resonant GKZ systems that arise in mirror symmetry and other contexts in algebraic geometry. We apply the mixed-twist construction in two distinct arenas, for constructing multi-parameter families of lattice polarized K3 surfaces, and the mirror family of Calabi-Yau $n$-folds $Y^{(n)}$ described above. Our approach in each case differs in somewhat major ways.
\par In the former, we utilize the geometry of K3 surface constructed through the mixed-twist construction to connect to some known results in the literature, allowing us to determine the monodromy group. In particular, since the K3 surface is presented explicitly as a Jacobian elliptic fibration, the mixed-twist construction that we apply to a certain family of elliptic curves with rational total space coincides with the well known quadratic twist construction in the theory of elliptic surfaces. From the perspective of lattice polarizations, this construction is nontrivial. We prove that the new family of K3 surfaces is birationally equivalent to a family of double-sextic K3 surfaces, obtained from the minimal resolution of a double cover of $\mathbb{P}^2$ branched along six lines (for example, studied in  \cite{MR973860, MR1073363, MR1136204,MR4015343, MR3712162}). From here, we identify the lattice polarization $L$ for the family, and determine the global monodromy group, and the Picard-Fuchs system, the latter two being determined by the Aomoto-Gel’fand system $E(3,6)$, as studied in \cite{MR973860, MR1073363, MR1136204}. In particular, this system is a multi-parameter resonant GKZ hypergeometric system. We naturally determine the parameter space of this family as the complement of the singular locus of the fibration. Morevover, the structure of the fibration allows us to consider natural sub-varieties of the parameter space of double-sextics where the Picard-Fuchs system restricts to known lower-rank systems of resonant hypergeometric type. In each case, the global monodromy group is determined by connecting our family to known results in K3 geometry. We then show that these restrictions lead to extensions of the lattice polarization in a chain of even, indefinite, two-elementary lattices.  In this way, we are able to unify central analytical aspects for resonant generalized hypergeometric functions with geometric and lattice theoretic investigations by Hoyt \cite{MR1023921,MR1013162} and Hoyt and Schwarz \cite{MR1877757}.
\par In the second case, we look at an application of the mixed-twist construction in the context of the mirror families for the deformed Fermat pencils as outlined above. In fact, in this context the mixed-twist construction returns the mirror family of Calabi-Yau $n$-folds in $\mathbb{P}^{n+1}$ fibered by mirror Calabi-Yau $(n-1)$-folds. In this framework, the set of periods generates a set of resonant GKZ data, which makes the analysis of the behavior of the family near regular singular points quite difficult \cite{MR1672077}. However, we show that the mixed-twist construction also generates a second set of \emph{non-resonant} GKZ data associated with the holomorphic periods, which allows us to compute the explicit monodromy matrices for the mirror families. This second part generalizes work of Chen et al. \cite{MR2369490} where the authors constructed the monodromy group of the Picard-Fuchs differential equations  associated with the one-parameter families of Calabi-Yau threefolds from Doran and Morgan ~\cite{MR2282973}. 
\par We remark that the Picard-Fuchs operators for the families of mirror Calabi-Yau $n$-folds has been known since at least the work of Corti \& Golyshev \cite{MR2824960}. Our approach in this article is novel in the sense that it is inpired by the physics - in particular, by connections between effective Yang-Mills gauge theory (i.e., Seiberg-Witten theory) and string compactifications on Calabi-Yau varieties. The mixed-twist construction offers a potential to connect computations in these two realms, by geometrizing a link between families of elliptic curves and their Picard-Fuchs operators, and families of Calabi-Yau varieties and their Picard-Fuchs operators via the iterative period relation described above. In addition, the mixed-twist construction provides a mechanism by which to construct transcendental cycles on Calabi-Yau varieties. This allows for the description of the period integrals in terms of $\mathcal{A}$-hypergeometric functions. This approach was utilized, for example, by Clingher, Doran, \& Malmendier in \cite{MR3767270} to obtain a description of the periods of so-called generalized Kummer surfaces in terms of Appell's bivariate $F_2$ hypergeometric function.  
\par Our approach in this article is summarized as follows: in the first part we construct and analyze a family that generalizes the family of K3 surfaces whose rank-19 polarizing lattice is $H \oplus D_{16}(-1)  \oplus A_1(-1)$ and whose Picard-Fuchs operator is the hypergeometric differential equation for $\hpgo{3}{2}(\frac{1}{2},\frac{1}{2},\frac{1}{2};1,1 | \; \cdot \; )$. The generalization considered is a four-dimensional family of K3 surfaces whose polarizing lattice is $H \oplus D_{10}(-1) \oplus D_4(-1) \oplus A_1(-1)$, and whose Picard-Fuchs system is the Aomoto-Gel’fand system $E(3,6)$. In the second part we compute the monodromy matrices for the families of Calabi-Yau $(n-1)$-folds that extend the family of K3 surface whose rank-19 polarizing lattice is  $H \oplus E_8(-1) \oplus E_8(-1) \oplus \langle -4 \rangle$ and whose Picard-Fuchs operator is the hypergeometric differential equation for $\hpgo{3}{2}(\frac{1}{4},\frac{1}{2},\frac{3}{4};1,1 | \; \cdot \; )$. The generalization considered are the one-dimensional mirror families of deformed Fermat pencils whose Picard-Fuchs operator is the hypergeometric differential equation for  $\hpgo{n}{n-1}(\frac{1}{n+1},\dots,\frac{n}{n+1};1,\dots,1 | \; \cdot \; )$.  The main results of the two parts are Theorem~\ref{theorem1} and Theorem~\ref{MirrorRecursive2}, respectively.
\par This article is organized as follows. In \S 2 we review relevant background material, which includes multi-parameter Weierstrass models associated with families of Jacobian elliptic fibrations and their multivariate Picard-Fuchs operators. We also recall the fundamental definition of a generalized functional invariant and its relation to the mixed twist construction.  In \S 3 we use the mixed-twist construction to obtain a multi-parameter family of K3 surfaces of Picard rank $\rho \ge 16$. Upon identifying a particular Jacobian elliptic fibration on its general member, we find the corresponding lattice polarization, the parameter space, and the Picard-Fuchs system for the family with its global  monodromy group. We construct a sequence of restrictions that lead to extensions of the polarization keeping the polarizing lattice two-elementary. We show that the Picard-Fuchs operators under these restrictions coincide with well-known hypergeometric systems, the Aomoto-Gel'fand $E(3,6)$ system (for $\rho =17$), Appell's $F_2$ system (for $\rho=18$), and Gauss' hypergeometric functions of type $\hpgo{3}{2}$ (for $\rho=19$). This allows us to determine the global monodromy groups of each family. Finally, we will show in \S 4 that the mixed-twist construction produces for each mirror family a non-resonant GKZ system for which a basis of solutions in the form of absolutely convergent Mellin-Barnes integrals exists whose monodromy is then computed explicitly. 
\subsubsection*{Acknowledgement} We thank the referees for their many helpful comments, suggestions, and corrections. This work was partially supported by a grant from the Simons Foundation.
\section{Elliptic fibrations and the mixed-twist construction}
\label{review}
In this section we give some well-known results on Weierstrass models and their period integrals. We also review the generalized functional invariant.

\subsection{Weierstrass models and their Picard-Fuchs Operators}
\label{ss-weierstrass}
We begin by recalling some basic notions of elliptic fibrations and the associated Weierstrass models. Let $X$ and $S$ be normal complex algebraic varieties and $\pi : X \to S$ an elliptic fibration, that is, $\pi$ is proper surjective morphism with connected fibers such that the general fiber is a nonsingular elliptic curve. Moreover, we assume that $\pi$ is smooth over an open subset $S_0 \subset S$, whose complement in $S$ is a divisor with at worst normal crossings. Thus, the local system $H^i_0 := R^i\pi_* \underline{\mathbb{Z}}_X {|}_{S_0}$ forms a variation of Hodge structure over $S_0$. 
\par Elliptic fibrations possess the following canonical bundle formula: on $S$, the \emph{fundamental line bundle} denoted $\mathcal{L} := (R^1\pi_*\mathcal{O}_X)^{-1}$ and the canonical bundles $\boldsymbol\omega_X := \wedge^{\mathrm{top}}\, T^{*(1,0)} X$, $\boldsymbol\omega_S := \wedge^{\mathrm{top}} \, T^{*(1,0)} S$ are related by 
\begin{equation} 
\label{canonicalbundle} 
\boldsymbol \omega_X \cong \pi^*(\boldsymbol\omega_S \otimes \mathcal{L}) \otimes \mathcal{O}_X(D), 
\end{equation} 
where $D$ is a certain effective divisor on $X$ depending only on divisors on $S$ over which $\pi$ has multiple fibers, and divisors on $X$ giving $(-1)$-curves of $\pi$. When $\pi : X \to S$ is a \emph{Jacobian elliptic fibration}, that is, when there is a section $\sigma : S \to X$, the case of multiple fibers is prevented. We may avoid the presence of $(-1)$-curves in the following way: For $X$ an elliptic surface, we assume that the fibration is relatively minimal, meaning that there are no $(-1)$-curves in the fibers of $\pi$. When $X$ is an elliptic threefold, we additionally assume that no contraction of a surface is compatible with the fibration. 
\par Assuming these minimality constraints, we have $D=0$, thus the canonical bundle formula (\ref{canonicalbundle}) simplifies to $\boldsymbol\omega_X \cong \pi^*(\boldsymbol\omega_S \otimes \mathcal{L})$. In particular, for $\mathcal{L} \cong \boldsymbol\omega_S^{-1}$ we obtain $\boldsymbol\omega_X \cong \mathcal{O}_X$. Recall that $X$ is a Calabi-Yau manifold if $\boldsymbol\omega_X \cong \mathcal{O}_X$ and $h^i(X,\mathcal{O}_X) =0$ for $0 < i < n=\dim(X)$. In this article we will be concerned with Jacobian elliptic fibrations on Calabi-Yau manifolds. It is well known that for $X$ an elliptic Calabi-Yau threefold, the base surface can have at worst log-terminal orbifold singularities.  We will take the base surface $S$ to be a Hirzebruch surface $\mathbb{F}_k$ (or its blowup).
\par It is well known that Jacobian elliptic fibrations admit \emph{Weierstrass models}, i.e., given a Jacobian elliptic fibration $\pi: X \to S$ with section $\sigma : S \to X$, there is a complex algebraic variety $W$ together with a proper, flat, surjective morphism $\hat{\pi} : W \to S$ with canonical section $\hat{\sigma} : S \to W$ whose fibers are irreducible cubic plane curves, together with a birational map $X \dashrightarrow W$ compatible with the sections $\sigma$ and $\hat{\sigma}$; see \cite{MR690264}. The map from $X$ to $W$ blows down all components of the fibers that do not intersect the image $\sigma(S)$. If $\pi : X \to S$ is relatively minimal, the inverse map $W \dashrightarrow X$ is a resolution of the singularities of $W$. 
\par A Weierstrass model is constructed as follows: given a line bundle $\mathcal{L} \to S$,  and sections $g_2$, $g_3$ of $\mathcal{L}^4$, $\mathcal{L}^6$ such that the discriminant $\Delta= g_2^3 -27g_3^2$ as a section of $\mathcal{L}^{12}$ does not vanish, define a $\mathbb{P}^2$-bundle $p: \mathbf{P} \to S$ as $\mathbf{P} := \mathbb{P}\left(\mathcal{O}_S \oplus \mathcal{L}^2 \oplus \mathcal{L}^3\right)$ with $p$ the natural projection.  Moreover, let $\mathcal{O}_{\mathbf{P}}(1)$ be the tautological line bundle.  Denoting $x,y$ and $z$ as the sections of $\mathcal{O}_{\mathbf{P}}(1) \otimes \mathcal{L}^2$, $\mathcal{O}_P(1) \otimes \mathcal{L}^3$ and $\mathcal{O}_{\mathbf{P}}(1)$ that correspond to the natural injections of $\mathcal{L}^2$, $\mathcal{L}^3$ and $\mathcal{O}_S$ into $\pi_* \mathcal{O}_{\mathbf{P}}(1) = \mathcal{O}_S \oplus \mathcal{L}^2 \oplus \mathcal{L}^3$, the Weierstrass model $W$ from above is given by the the sub-variety of $\mathbf{P}$ defined by the equation 
\begin{equation}
\label{weierstrass_model}
y^2 z=4x^3 - g_2 x z^2-g_3z^3.
\end{equation} 
The canonical section $\sigma: S \to W$ is given by the point $[x : y : z] = [0 : 1 : 0]$ in each fiber, such that $\Sigma := \sigma(S) \subset W$ is a Cartier divisor whose normal bundle is isomorphic to the fundamental line bundle $\mathcal{L}$ via $p_*\mathcal{O}_{\mathbf{P}}(-\Sigma) \cong \mathcal{L}$. It follows that $W$ inherits the properties of normality and Gorenstein if $S$ possesses these. Thus, the canonical bundle formula ~(\ref{canonicalbundle}) reduces to 
\begin{equation}
\boldsymbol\omega_W = \pi^*\left(\boldsymbol\omega_S \otimes \mathcal{L}\right).
\end{equation}
The Jacobian elliptic fibration $p: W \to S$ then has a Calabi-Yau total space if $\mathcal{L} \cong \boldsymbol \omega_S^{-1} = \mathcal{O}_S(-K_S)$ (misusing notation slightly to denote the projection map $p$ the as the projection from the ambient $\mathbb{P}^2$-bundle). 
\par For a Jacobian elliptic fibration $X$ the canonical bundle $\boldsymbol\omega_X$ is determined by the discriminant $\Delta = g_2^3 -27g_3^2$.  For example, if $\pi : X \to S$ is a Jacobian elliptic fibration for a smooth algebraic surface $X$ and $S = \mathbb{P}^1$ with homogeneous coordinates $[t:s]$, then $X$ is a rational elliptic surface if the $\Delta$ is a homogeneous polynomial of degree 12 (meaning that $\mathcal{L} = \mathcal{O}(1)$),  and $X$ is a K3 surface when $\Delta$  is a homogeneous polynomial of degree 24 (meaning that $\mathcal{L} = \mathcal{O}(2)$); these results follow readily from adjunction and Noether's formula. The nature of the singular fibers and their effect on the canonical bundle was established by the seminal work of Kodaira ~\cite{MR205280, MR228019, MR0184257}.  
\par Of particular interest in this article are multi-parameter families of elliptic Calabi-Yau $n$-folds over a base $B$, a quasi-projective variety of dimension $r$, denoted by $\pi : ~X \to B$. Hence, each $X_p = \pi^{-1}(p)$ is a compact, complex $n$-fold with trivial canonical bundle. Moreover, each $X_p$ is elliptically fibered with section over a fixed normal variety $S$. This means that we have a multi-parameter family of minimal Weierstrass models $p_b : W_b \to S$ representing a family of Jacobian elliptic fibrations $\pi_b : X_b \to S$. We denote the collective family of Weierstrass models as $p: W \to B$. 
\par Working within affine coordinates for $B$ and $S$ we set $u=(u_1,\dots, u_{n-1}) \in \mathbb{C}^{n-1} \subset S$ and $b=(b_1,\dots, b_r) \in \mathbb{C}^{r} \subset B$. We then may write the Weierstrass model $W_b$ in the form 
\begin{equation}
\label{multi-parameter_weierstrass_model}
y^2=4x^3-g_2(u, b) x-g_3(u, b),
\end{equation}
where for each fiber we have chosen the affine chart of $W_b$ given by $z=1$ in Equation~(\ref{weierstrass_model}).
\par Part of the utility of a Weierstrass model is the explicit construction of the holomorphic $n$-form on each $X_b$, up to fiberwise scale, allowing for the detailed study of the Picard-Fuchs operators underlying a variation of Hodge structure.  In fact, consider the holomorphic sub-system $H \to B$ of the local system $V = R^n \pi_* \underline{\mathbb{C}}_X \to B$, whose fibers are given as the line $H^0(\omega_{X_b}) \subset H^n(X_b,\mathbb{C})$.  Here, $\underline{\mathbb{C}} \to X$ is the constant sheaf whose stalks are $\mathbb{C}$.  Griffiths showed~\cite{MR233825,MR229641,MR282990, MR258824} that  $\mathcal{V} = V \otimes_{\mathbb{C}} \mathcal{O}_B$ is a vector bundle carrying a canonical flat connection $\nabla$, the Gauss-Manin connection. A meromorphic section of $\mathcal{H} = H \otimes_{\mathbb{C}} \mathcal{O}_B \subset \mathcal{V}$ is given fiberwise by the holomorphic $n$-form $\eta_b \in H^0(\boldsymbol\omega_{X_b}) \subset H^n(X_b,\mathbb{C})$
\begin{equation}
\label{holomorphic_n-form}
\eta_b=du_1 \wedge \cdots \wedge du_{n-1} \wedge \frac{dx}{y},
\end{equation}
where we denote the collective section as $\eta \in \Gamma(\mathcal{V},B)$. It is natural to consider local parallel sections of the dual bundle $\mathcal{H}^*=H^* \otimes_{\mathbb{C}}\mathcal{O}_B$, where $H^*$ is the local system dual to $H$; these are generated by transcendental cycles $\Sigma_b \in H_n(X_b,\mathbb{R})$ that vary continuously with $b \in B$, writing the collective section as $\Sigma \in \Gamma(\mathcal{V}^*,B)$. The sections are covariantly constant since the local system $V = R^n \pi_* \underline{\mathbb{C}}_X$ is locally topologically trivial, and thus local sections of the dual $V^*$ are as well. Utilizing the natural fiberwise de Rham pairing
\begin{equation*}
\langle \Sigma_b, \eta_b \rangle = \oint_{\Sigma_b} \eta_b,
\end{equation*}
we obtain the period sheaf $\Pi \to B$, whose stalks are given by the local analytic function $b \mapsto \omega(b) = \langle \Sigma_b, \eta_b \rangle$. The function $\omega(b)$ is called a \emph{period integral} (over $\Sigma_b$) and satisfies a system of coupled linear PDEs in the variables $b_1,\dots, b_r$ -- the so called \emph{Picard-Fuchs} system - whose rank is that of the period sheaf $\Pi \to B$, or the number of linearly independent period integrals of the family.
\par Given the affine local coordinates $(b_1,\dots, b_r) \in \mathbb{C}^r \subset B$, fix the meromorphic vector fields $\partial_j = \partial / \partial b_j$ for $j=1,\dots, r$. Then each $\partial_j$ induces a covariant derivative operator $\nabla_{\partial_j}$ on $\mathcal{V}$. Since $\nabla$ is flat, the curvature tensor $\Omega=\Omega_\nabla$ vanishes, and hence, for all meromorphic vector fields $U,V$ on $B$ we have 
\begin{equation*}
\Omega(U,V) = \nabla_U\nabla_V - \nabla_V\nabla_U - \nabla_{[U,V]}=0.
\end{equation*}
Substituting in the commuting coordinate vector fields $\partial_i,\partial_j$, we conclude 
\begin{equation*}
\nabla_{\partial_i}\nabla_{\partial_j}=\nabla_{\partial_j}\nabla_{\partial_i}.
\end{equation*}
This integrability condition is crucial in obtaining a system of PDEs from the Gauss-Manin connection. Since $\mathcal{V}$ has rank $m= \dim H^n(X_b,\mathbb{C})$, each sequence of parallel sections $\nabla^{i_1}_{\partial_{k_1}} \cdots \nabla^{i_{\hat{m}}}_{\partial_{k_r}} \eta$, for $i_1+\cdots + i_{\hat{m}}=0,1,\dots,\hat{m}$ and $1 \leq k_1, \dots, k_r \leq r$ form the linear dependence relations
\begin{equation*}
\sum_{i_1+\cdots + i_{\hat{m}}=0}^{\hat{m}} \, \sum_{k_1,\dots,k_r=1}^r a^{k_1 \cdots k_r}_{i_1 \cdots i_{\hat{m}}}(b) \, \nabla^{i_1}_{\partial_{k_1}} \cdots \nabla^{i_{\hat{m}}}_{\partial_{k_r}}\, \eta =0
\end{equation*}
for some integer $0 < \hat{m} \leq m$, where $a^{k_1 \cdots k_r}_{i_1 \cdots i_{\hat{m}}}(b)$ are meromorphic. Here, it is understood that $\nabla^0 = \mathrm{id}$. As $\nabla$ annihilates the transcendental cycle $\Sigma$ and is compatible with the pairing $\langle \Sigma, \eta \rangle$, we may ``differentiate under the integral sign" to obtain
\begin{equation*}
\frac{\partial}{\partial b_j} \omega(b) = \frac{\partial}{\partial b_j} \oint_{\Sigma} \eta = \oint_{\Sigma} \nabla_{\partial_j} \eta.
\end{equation*}
 It follows that the period integral $\omega(b)$ satisfies the system of linear PDEs of rank $\mathsf{r} \geq 1$, given by
\begin{equation}
\label{PF_equations}
\sum_{i_1+\cdots + i_{\hat{m}}=0}^{\hat{m}} \, \sum_{k_1,\dots,k_r=1}^r a^{k_1 \cdots k_r}_{i_1 \cdots i_{\hat{m}}}(b) \frac{\partial^{i_1+\dots +i_{\hat{m}}}}{\partial^{k_1}b_{k_1} \cdots \partial^{k_r}b_{k_r}}\omega(b)=0.
\end{equation}
Equation~(\ref{PF_equations}) is the \emph{Picard-Fuchs system} of the multi-parameter family $\pi : X \to B$ of Calabi-Yau $n$-folds. The resulting system is then known to be a linear Fuchsian system, i.e., the system with at worst regular singularities. This is due to analytical results of Griffiths \cite{MR282990} and Deligne \cite{deligne_equations_1970} who utilized Hironaka's resolution of singularities \cite{hironaka_resolution_1964,hironaka_resolution_1964-1} to estimate the growth of solutions of the system. 
\par The rank $\mathsf{r}$ and order $\hat{m}$ of the system depends on the parameter space $B$ and algebro-geometric data of the generic fiber $X_b$. For example,  let $\pi : X \to B$ be a family of Jacobian elliptic K3 surfaces which is polarized by a lattice\begin{footnote}{For the definition of lattice polarized K3 surface, see \S \ref{ss-latice_polarization}.}\end{footnote} $L$ of rank $\rho \leq 18$ such that $B$ defines an $n=20-\rho$ dimensional family of L-polarized K3 surfaces. By results due to Dolgachev \cite{MR1420220}, there is a coarse moduli space $\mathcal{M}_L$ of all lattice polarized K3 surfaces of dimension $n$; in this case, we are requiring that $B$ be a top dimensional family of $L$-polarized K3 surfaces. It then follows from the general program of Sasaki and Yoshida \cite{sasaki_linear_1989} on orbifold uniformizing differential equations that the Picard-Fuchs system~(\ref{PF_equations}) is a linear system of order $\hat{m}=2$ and rank $\mathsf{r}=n+2$ in $n$ variables, the latter coming from the local coordinates in $B$. Naturally, there are sub-loci of such parameter spaces $B$ where the lattice polarization extends to higher Picard rank and the rank of the Picard-Fuchs system drops accordingly. This behavior was studied, for example, by Doran et al.~in \cite{MR3524237}, and coined the \emph{differential rank-jump property} therein.  In the sequel, we will analyze it by studying corresponding Weierstrass model $p : W \to B$. Moreover, we will see that the Picard-Fuchs system can be explicitly computed from the geometry of the elliptic fibrations and the presentation of the associated period integrals as generalized Euler integrals using GKZ systems ~\cite{MR1080980}.
\par It is commonplace in the literature to study the Picard-Fuchs equations of one parameter families of Calabi-Yau $n$-folds; in this case, the base $B$ is a punctured complex plane with local affine coordinate $t \in \mathbb{C} \subset B$, and an analogous construction leads to a regular Fuchsian ODE of order $\leq m$ with $m=  \dim H^n(X_t,\mathbb{C})$ for the general fiber $X_t$. In the construction of Doran and Malmendier \cite{MR4069107}, this is the central focus, with $B = \mathbb{P}^1 - \{0,1,\infty\}$ and $B = \mathbb{P}^1 - \{0,1, p, \infty\}$. We will show that the restriction of the multi-parameter Picard-Fuchs system~(\ref{PF_equations}) above leads to the Picard-Fuchs ODE operators and families of lattice polarized K3 surfaces of Picard rank $\rho=19$, for example the mirror partners of the classic deformed Fermat quartic K3. 
\subsection{The generalized functional invariant} 
\label{ss-twist}
We first recall the \emph{generalized functional invariant} of the mixed-twist construction studied by Doran and Malmendier \cite{MR4069107}, first introduced by Doran \cite{MR1877754}.  A generalized functional invariant is a triple $(i, j, \alpha)$ with $i, j \in \mathbb{N}$ and $\alpha \in \left \{\frac{1}{2},1 \right \}$ such that $1 \leq i, j \leq 6$. To this end, the generalized functional invariant encodes a 1-parameter family of degree $i + j$ covering maps $\mathbb{P}^1 \to \mathbb{P}^1$,  which is totally ramified over $0$, ramified to degrees $i$ and $j$ over $\infty$, and simply ramified over another point $\tilde{t}$.  For homogeneous coordinates $[v_0 : v_1] \in \mathbb{P}^1$, this family of maps (parameterized by $\tilde{t} \in \mathbb{P}^1 - \{0,1,\infty\}$) is given by 
\begin{equation}
\label{eqn-generalized_functional_invariant}
[v_0,v_1] \mapsto [c_{ij}v_1^{i+j}\tilde{t} : v_0^i(v_0+v_1)^{j}],
\end{equation}
for some constant $c_{i j} \in \mathbb{C}^\times$.  For a family $\pi : X \to B$ with Weierstrass models given by Equation~(\ref{multi-parameter_weierstrass_model}) with complex $n$-dimensional fibers and a generalized functional invariant $(i, j, \alpha)$ such that 
\begin{equation}
\label{eqn-functional_invariant_conditions}
0 \leq \mathrm{deg}_t(g_2) \leq \mathrm{min} \left ( \frac{4}{i},\frac{4\alpha}{j} \right ), \hspace{5mm} 0 \leq \mathrm{deg}_t(g_3) \leq \mathrm{min} \left ( \frac{6}{i},\frac{6\alpha}{j} \right ), 
\end{equation}
Doran and Malmendier showed that a new family $\tilde{\pi} : \tilde{X} \to B$ can be constructed such that the general fiber $\tilde{X}_{\tilde{t}} = \tilde{\pi}^{-1}(\tilde{t})$ is a compact, complex $(n+1)$-manifold equipped with a Jacobian elliptic fibration over $\mathbb{P}^1 \times S$.  In the coordinate chart $\{ [ v_0 : v_1 ], (u_1,\dots, u_{n-1})\} \in \mathbb{P}^1 \times S$ the family of Weierstrass models $W_{\tilde{t}}$ is given by 
\begin{equation}
\label{eqn-twisted_weierstrass}
\begin{split}
\tilde{y}^2  = 4\tilde{x}^3 & - g_2 \left ( \frac{c_{ij} \tilde{t} v_1^{i+j}}{v_0^i(v_0 + v_1)^j},u \right ) v_0^4v_1^{4-4\alpha}(v_0+v_1)^{4\alpha} \tilde{x} \\
& - g_3 \left ( \frac{c_{ij} \tilde{t} v_1^{i+j}}{v_0^i(v_0 + v_1)^j},u \right ) v_0^6v_1^{6-6\alpha} (v_0+v_1)^{6\alpha}
\end{split}
\end{equation}
with $c_{i j} = (-1)^i i^i j^j / (i+j)^{i+j}$.  The new family is called the \emph{twisted family with generalized functional invariant $(i, j, \alpha)$ of $\pi : X \to B$}. It follows that conditions (\ref{eqn-functional_invariant_conditions}) guarantee that the twisted family is minimal and normal if the original family is. Moreover, they showed that if the Calabi-Yau condition is satisfied for the fibers of the twisted family if it is satisfied for the fibers of the original. 
\par The twisting associated with the generalized functional invariant above is referred to as the \emph{pure twist construction}; we may extend this notion to that of a \emph{mixed twist construction}. This means that one combines a pure twist from above with a rational map $B \to B$, thus allowing one to change locations of the singular fibers and ramification data. This was studied in \cite[Sec.~8]{MR4069107} for linear and quadratic base changes. We may also perform a multi-parameter version of the mixed twist construction for a generalized functional invariant $(i, j, \alpha)=(1,1,1)$. For us, it will be enough to consider the two-parameter family of ramified covering maps given by
\begin{equation}
\label{eqn-two_parameter}
[v_0 : v_1] \mapsto [4 a v_0 (v_0 + v_1) + (a - b)v_1^2 : 4 v_0 (v_0 + v_1)],
\end{equation}
such that for $a, b \in \mathbb{P}^1 - \{0,1,\infty\}$ with $a \not = b$ the map in Equation~(\ref{eqn-two_parameter}) is totally ramified over $a$ and $b$.   We will apply the mixed twist construction to certain (families of) rational elliptic surfaces $X \to \mathbb{P}^1$.  In \cite[Sec.~5.5]{MR4069107} the authors showed that the twisted family with generalized functional invariant $(1,1,1)$ in this case is birational to a quadratic twist family of $X \to \mathbb{P}^1$. We will explain the relationship in more detail and utilize it in the construction of the associated Picard-Fuchs operators in the next section.
\section{A multi-parameter family of K3 surfaces}
\label{multi}
In this section, we use the mixed-twist construction to obtain a multi-parameter family of K3 surfaces of Picard rank $\rho \ge 16$. Upon identifying a particular Jacobian elliptic fibration on its general member, we find the corresponding lattice polarization and the Picard-Fuchs system using the results from \S \ref{ss-weierstrass}. We construct a sequence of restrictions on the parameter space that lead to extensions of the lattice polarization, while keeping the polarizing lattice two-elementary. 
\par Moreover, we show that the Picard-Fuchs operators under these restrictions coincide with well-known hypergeometric systems, the Aomoto-Gel'fand $E(3,6)$ system (for $\rho =16,17$), Appell's $F_2$ system (for $\rho=18$), and Gauss' hypergeometric functions of type $\hpgo{3}{2}$ (for $\rho=19$). Each such Picard-Fuchs system forms a \emph{resonsant} GKZ hypergeometric system. We also determine the corresponding monodromy group for each family. %
\subsection{Quadratic twists and double-sextics}
\label{ss-quadratic_twist}
A two-parameter family of rational elliptic surfaces $S_{c, d} \to \mathbb{P}^1$ is given by the affine Weierstrass model
\begin{equation}
\label{eqn-rational}
{y}^{2}=4x^3-g_2(t)x-g_3(t),
\end{equation}
where $g_2(t)$ and $g_3(t)$ are the following polynomials of degree four and six, respectively,
\begin{equation*}
\begin{split}
g_2 & =\frac{4}{3}\left(t^4 - (2c + d + 1)t^3 + (c^2 + cd + d^2 + 2c - d + 1)t^2 - c(c - d + 2)t + c^2\right), \\
g_3 & =\frac{4}{27}\left(t^2 - (c - d + 2)t + 2c\right)\left(t^2 - (c + 2d - 1)t - c\right)\left(2t^2 - (2c + d + 1)t + c\right),
\end{split}
\end{equation*}
where $t$ is the affine coordinate on the base curve. Assuming general parameters $c, d$, Equation (\ref{eqn-rational}) defines a rational elliptic surface  with 6 singular fibers of Kodaira type $I_2$ over $t=0,1,\infty, c, c+d,$ and $c/(d-1)$.  We have the following:
\begin{lemma}
\label{prop-ratnet}
The rational elliptic surface $S=S_{c, d}$ in Equation~(\ref{eqn-rational}) is birationally equivalent to the twisted Legendre pencil
\begin{equation}
\label{eqn-generalized_legendre}
\tilde{y}^2 = \tilde{x}(\tilde{x}-1)(\tilde{x}-t)(t-c-d\tilde{x}).
\end{equation}
\end{lemma}
\begin{proof}
By direct computation using the transformation:
\begin{equation*}
x=\frac{3t \left( t-c \right)}{3\tilde{x}+{t}^{2}+\left(d+1-c\right)t-c}, \quad
y=\frac{3\tilde{y}t \left( t-c \right)}{2\left(3\tilde{x}+{t}^{2}+\left(d+1-c\right)t-c\right)^2}.
\end{equation*}
\end{proof}
\par A quadratic twist applied to a rational elliptic surface can be identified with Doran and Malmendier's mixed-twist construction with generalized functional invariant $(i, j, \alpha)=(1,1,1)$. The two-parameter family of ramified covering maps in Equation (\ref{eqn-two_parameter}) is totally ramified over $a, b \in \mathbb{P}^1 - \{0,1,\infty\}$ with $a \not = b$. We apply the mixed-twist construction to the rational elliptic surface $S_{c,d}$:
\begin{proposition}
\label{lem:construction}
The mixed-twist construction with generalized functional invariant $(i, j,\alpha)=(1,1,1)$ applied to the rational  elliptic surface in Equation~(\ref{eqn-rational}) yields the family  of Weierstrass models
\begin{equation}
\label{eqn-four_parameter}
\hat{y}^2=4\hat{x}^3-(t-a)^2(t-b)^2g_2(t)\hat{x}-(t-a)^3(t-b)^3g_3(t).
\end{equation}
The family is birationally equivalent to 
\begin{equation}
\label{eqn-extended_legendre}
 y^2=x(x-1)(x-t)(t-a)(t-b)(t-c-dx).
\end{equation}
Over the four-dimensional parameter space
\beq
\label{eqn:moduli}
 \mathcal{M} = \Big\{ (a, b, c, d) \in \mathbb{C}^4 \ \Big|  \ a \not = b, \  (c, d) \not = (a, 0),  (b, 0), (0, 1) \Big\} \,,
\eeq 
Equation~(\ref{eqn-extended_legendre}) defines a family of Jacobian elliptic K3 surfaces $\mathbf{X}_{a, b, c, d} \to \mathbb{P}^1$. 
\end{proposition}
\begin{proof}
In affine base coordinates $[v : 1] \in \mathbb{P}^1$, the map $f : \mathbb{P}^1 \to \mathbb{P}^1$ from the mixed-twist construction with generalized functional invariant $(i, j, \alpha)=(1,1,1)$ in Equation~(\ref{eqn-two_parameter}) is given by
\begin{equation*}
f(v)=a + \frac{a-b}{4v(v+1)}.
\end{equation*}
The pullback of the Weierstrass model for the two-parameter family of the rational elliptic surfaces in Equation~(\ref{eqn-rational}) along the map $t=f(v)$ is easily checked to yield the four-parameter family in Equation~(\ref{eqn-four_parameter}). Equation~(\ref{eqn-extended_legendre}) follows from a direct computation, with the following transformation:
\begin{equation*}
\begin{split}
\hat{x}=\frac{3t \left( t-a \right)  \left( t-b \right)\left( t-c \right)}{3x+\left( t-a \right)  \left( t-b \right)\left({t}^{2}+\left(d+1-c\right)t-c\right)}, \\
\hat{y}=\frac{3yt \left( t-a \right)  \left( t-b \right)\left( t-c \right)}{2\left(3x+\left( t-a \right)  \left( t-b \right)\left({t}^{2}+\left(d+1-c\right)t-c\right)\right)^2}.
\end{split}
\end{equation*}
One checks that for parameters in $\mathcal{M}$ the minimal resolution of Equation~(\ref{eqn-four_parameter}) defines a Jacobian elliptic K3 surfaces $\mathbf{X}_{a, b, c, d} \to \mathbb{P}^1$. In fact, Equation~(\ref{eqn-four_parameter}) is a minimal Weierstrass equation of a K3 surface if and only if $a \not = b$ and $(c, d) \not = (a, 0),  (b, 0), (0, 1)$. 
\end{proof}
A direct computation for the Weierstrass model yields the following:
\begin{lemma}
\label{lem:Jac16}
Equation~(\ref{eqn-four_parameter}) defines a Jacobian elliptic fibration $\pi : \mathbf{X}  \to ~\mathbb{P}^1$ on a general $\mathbf{X} = \mathbf{X}_{a, b, c, d}$ with two singular fibers of Kodaira type $I_0^*$ over $t=a, b$, six singular fibers of Kodaira type $I_2$ over $t=0,1,\infty, c, c+d,$ and $c/(d-1)$, and the Mordell Weil group $\mathrm{MW}(\mathbf{X},\pi) = (\mathbb{Z}/2\mathbb{Z})^2$.  
\end{lemma}
\par Equation~(\ref{eqn-extended_legendre}) provides a model for the K3 surfaces $\mathbf{X}$ as double covers of the projective plane branched on the union of six lines. In general, we call  a K3 surface $\mathcal{X}$ a \emph{double-sextic surface} if it is the minimal resolution of a double cover of the projective plane $\mathbb{P}^2$ branched along the union of six lines, which we denote by $\boldsymbol \ell =\{\ell_1,\dots,\ell_6\}$. In weighted homogeneous coordinates $[t_1 : t_2 : t_3 : z] \in \mathbb{P}(1,1,1,3)$ such a double-sextic is given by the equation
\begin{equation}
\label{eqn-six_lines}
z^2=\prod_{i=1}^6(a_{i1}t_1+a_{i2}t_2+a_{i3}t_3),
\end{equation}
where the lines $\ell_i =\left\{ [t_1 : t_2 : t_3] \; | \; a_{i1}t_1+a_{i2}t_2+a_{i3}t_3 = 0 \right\} \subset \mathbb{P}^2$ for parameters $a_{i j} \in \mathbb{C}$, $i=1,\dots,6$, $j=1,2,3$ are assumed to be general.  Let $A = (a_{i j}) \in \mathrm{Mat}(3,6;\mathbb{C})$ be the matrix whose entries are the coefficients encoding the six-line configuration $\boldsymbol \ell$. Let $M$ be the configuration space of six lines $\boldsymbol \ell$ whose minimal resolution is a K3 surface. Then isomorphic K3 surfaces are obtained if we act on elements $A \in M$ by matrices induced from automorphisms of $\mathbb{P}^2$ on the left and overall scale changes of each line $\ell_i \in \boldsymbol \ell$ on the right. Thus, we are led to consider the four-dimensional quotient space
\begin{equation}
\label{eqn:M6}
\mathcal{M}_6 = \mathrm{SL}(3,\mathbb{C}) \backslash M / (\mathbb{C}^*)^6,
\end{equation}
and $\mathcal{M}$ in Equation~(\ref{eqn:moduli}) can be identified with the open subspace of $\mathcal{M}_6$, given by elements $[A] \in \mathcal{M}_6$ of the form
\begin{equation*}
\begin{pmatrix}
1 & 1  & 1  & 0 & 0 & -d \\
0 & 1  & 0 & -1 & -1 & -1 \\
0 & 0  & 1  & a & b & c 
\end{pmatrix}
\end{equation*}
with $(a, b, c, d) \in \mathcal{M}$ and $t_1=x, t_2=-t, t_3=-1$.
\par The family of double-sextics in Equation~(\ref{eqn-six_lines}) has been studied in the literature, for example by Matsumoto \cite{MR1103969},  and Matsumoto et al. \cite{MR973860, MR1073363, MR1136204}. One takeaway from their work is that the family of double sextic K3 surfaces is, in many ways, analogous to the Legendre pencil of elliptic curves which is realized as double covers of $\mathbb{P}^1$ branching over four points. More recently, the double-sextic family $\mathcal{X}$ and closely related K3 surfaces have been studied in the context of string dualities \cite{MR3366121, MR3712162, MR4015343, MR2854198, MR4160930}.  In Clingher et al. \cite{MR4015343}, the authors showed that four different elliptic fibrations on $\mathcal{X}$ have interpretations in F-theory/heterotic string duality. Similar constructions are relevant to anomaly cancellations \cite{MR4092552}, studied by the authors of the present article. In \cite{MR4160930}, the authors classified all Jacobian elliptic fibrations on the Shioda-Inose surface associated with $\mathcal{X}$. Finally, Hosono et al. ~in \cite{MR4164174,MR4164174} constructed compactifications of $\mathcal{M}_6$ from GKZ data and toric geometry, suitable for the study of the Type IIA/Type IIB string duality. 
\subsection{Determination of the lattice polarization and monodromy}
\label{ss-latice_polarization}
In the following we will use the following standard notations for lattices: $L_1 \oplus L_2$ is orthogonal sum of the two lattices $L_1$ and $L_2$, $L(\lambda)$ is obtained from the lattice $L$ by multiplication of its form by $\lambda \in \mathbb{Z}$, $\langle R \rangle$ is a lattice with the matrix $R$ in some basis; $A_n$, $D_m$, and $E_k$ are the positive definite root lattices for the corresponding root systems,  $H$ is the unique even unimodular hyperbolic rank-two lattice.  A lattice $L$ is two-elementary if its discriminant group $A_L$ is a two-elementary abelian group, namely $A_L \cong (\mathbb{Z}/2\mathbb{Z})^\ell$ with $\ell$ being the minimal number of generators of the discriminant group $A_L$, also called the length of the lattice $L$.  Even, indefinite, two-elementary lattices $L$ are uniquely determined by the rank $\rho$, the length $\ell$, and the parity $\delta$ -- which equals $1$ unless the discriminant form $q_L(x)$ takes values in $\mathbb{Z}/2\mathbb{Z} \subset \mathbb{Q}/2\mathbb{Z}$ for all $x \in A_L$ in which case it is $0$; this is a result by Nikulin \cite[Thm.~4.3.2]{MR633160}.  
\par Let $\mathbf{X}$ be a smooth algebraic K3 surface over the field of complex numbers. Denote by $\operatorname{NS}(\mathbf{X})$ the N\'eron-Severi lattice of $\mathbf{X}$. This is known to be an even lattice of signature $(1,\rho_\mathbf{X}-1)$, where $p_\mathbf{X}$ denotes the Picard number of $\mathbf{X}$, with $1 \leq \rho_\mathbf{X} \leq 20$. In this context, a {\it lattice polarization} \cite{MR0357410,MR0429917,MR544937,MR525944,MR728992} on $\mathbf{X}$ is, by definition, a primitive lattice embedding $i \colon L \hookrightarrow \operatorname{NS}(\mathbf{X})$, with $i(L)$ containing a pseudo-ample class, i.e., a numerically effective class of positive self-intersection in the N\'eron-Severi lattice $\mathrm{NS}(\mathbf{X})$.  Here, $L$ is a choice of even lattice of signature $(1,\rho)$, with $ 1 \leq \rho \leq 20$ that admits a primitive embeddings into the K3 lattice $\Lambda_{K3}\cong H^{\oplus 3} \oplus E_8(-1)^{\oplus 2}$. Two $L$-polarized K3 surfaces $(\mathbf{X},i)$ and $(\mathbf{X}',i')$ are said to be isomorphic\footnote{Our definition of isomorphic lattice polarizations coincides with the one used by Vinberg \cite{MR2682724, MR2718942, MR3235787}. It is slightly more general than the one used in \cite[Sec.~1]{MR1420220}.},  if there exists an analytic isomorphism $\alpha \colon \mathbf{X} \rightarrow \mathbf{X}'$ and a lattice isometry  $ \beta \in O(L)$, such that $ \alpha^* \circ i' = i \circ \beta $, where $\alpha^*$ is the appropriate morphism at cohomology level. In general, $L$-polarized K3 surfaces are classified, up to isomorphism, by a coarse moduli space $\mathcal{M}_L$, which is known  \cite{MR1420220} to be a quasi-projective variety of dimension $20-\rho$.  A \emph{general} $L$-polarized K3 surface $(\mathbf{X},i)$ satisfies $i(L)=\operatorname{NS}(\mathbf{X})$.
\par We have the following result:
\begin{proposition}
\label{lemma-yoshida_surface}
Over $\mathcal{M}$ in Equation~(\ref{eqn:moduli}) the family
\begin{equation}
\label{eqn-extended_legendre2}
\mathbf{X}_{a, b, c, d}: \; \; y^2=x(x-1)(x-t)(t-a)(t-b)(t-c-dx).
\end{equation}
is a 4-dimensional family of $L$-polarized K3 surfaces where $L$ has rank 16 and the following isomorphic presentations:
\begin{equation}
\label{eqn:L}
\begin{split}
 L  \ \cong \ & H \oplus E_8(-1)  \oplus A_1(-1)^{\oplus 6}  \ \cong \ H \oplus E_7(-1) \oplus D_4(-1) \oplus A_1(-1)^{\oplus 3} \\
\cong \ & H \oplus D_6(-1) \oplus D_4(-1)^{\oplus 2} \ \cong \ H \oplus D_6(-1)^{\oplus 2} \oplus A_1(-1)^{\oplus 2} \\
 \ \cong \ & H \oplus D_{10}(-1)  \oplus A_1(-1)^{\oplus 4}  \cong \  H \oplus D_8(-1) \oplus D_4(-1) \oplus A_1(-1)^{\oplus 2}.
\end{split}
\end{equation}
In particular, $L$ is a primitive sub-lattice of the K3 lattice $\Lambda_{K3}$.
\end{proposition}
\begin{proof}
The general member of the family in Equation~(\ref{eqn-extended_legendre2}) is a double-sextic whose associated K3 surface has Picard number 16. A K3 surface $\mathbf{X}$ obtained as the minimal resolution of the double-sextic associated with a six-line configuration $\boldsymbol \ell$ in general position has the transcendental lattice $\mathrm{T}(\mathbf{X}) \cong  H(2) \oplus H(2) \oplus \langle -2 \rangle^{\oplus 2}$; see  \cite{MR1877757}. Accordingly, $\mathbf{X}$ has a N\'eron-Severi lattice given by a two-elementary lattice $L$ of rank $\rho=16$ such that $A_L \cong (\mathbb{Z}/2\mathbb{Z})^\ell$ with $\ell=6$.  From general lattice theory, it follows that  $L$ is the unique two-elementary lattice with $\rho=16$, $\ell=6$, $\delta=1$ (for $\rho=16$ the two-elementary lattice must have $\delta=1$; see \cite{MR633160}), and we obtain $L \cong  H \oplus E_8(-1)  \oplus A_1(-1)^{\oplus 6}$.
\par The family in Equation (\ref{eqn-four_parameter}) is birationally equivalent to the family in Equation~(\ref{eqn-extended_legendre}).  In turn, Lemma~\ref{prop-ratnet} identifies the family in Equation~(\ref{eqn-extended_legendre}) as a family of Jacobian elliptic K3 surfaces whose general member has the singular fibers $2 I_0^* + 6 I_2$ and the Mordell-Weil group $(\mathbb{Z}/2\mathbb{Z})^2$.  We then use results in \cite[Table~1]{MR2254405} to conclude that the general member of such a K3 surface $\mathbf{X}$ has the N\'eron-Severi lattice isomorphic to $H \oplus E_8(-1)  \oplus A_1(-1)^{\oplus 6}$. From \cite[Table~1]{MR2254405}  we also read off the isomorphic presentations of $L$ as the Jacobian elliptic fibrations supported on $\mathbf{X}$ with trivial Mordell Weil group. These elliptic fibrations prove that the lattice $L$ has the isomorphic presentations in Equation~(\ref{eqn:L}).
\end{proof}
\par The Picard-Fuchs system for the family can also be determined:
\begin{proposition}
\label{cor:periods}
Let $\Sigma \in \mathrm{T}(\mathbf{X})$ be a transcendental cycle on a general K3 surface $\mathbf{X} = \mathbf{X}_{a, b, c, d}$, $\eta_{\mathbf{X}}$ the holomorphic two-form induced by $dt \wedge dx/y$ in Equation~(\ref{eqn-extended_legendre2}), and $\omega = \oint_{\Sigma} \eta_{\mathbf{X}}$ a period. The Picard-Fuchs system for  $\mathbf{X}_{a, b, c, d}$, annihilating $\omega'= \sqrt{b(b-c)} \ \omega$, is the rank-six Aomoto-Gel'fand system $E(3,6)$ of \cite{MR973860, MR1073363} and \cite[\S 0.15]{MR1136204} in the variables
\begin{equation}
\label{eqn:params}
x_1 = \frac{a}{b}, \hspace{5mm} x_2= \frac{a-c}{b-c}, \hspace{5mm} x_3 = \frac{1}{b}, \hspace{5mm} x_4 = \frac{d}{b-c}.
\end{equation}
In particular, the Picard-Fuchs system is a resonant GKZ hypergeometric system.
\end{proposition}
\begin{proof} 
In \cite{MR973860}, a matrix $A = (a_{i j}) \in \mathrm{Mat}(3,6;\mathbb{C})$ was considered whose entries are the coefficients encoding a six-line configuration $\boldsymbol \ell$.  The authors used the action of $\mathrm{SL}(3,\mathbb{C})$ and $(\mathbb{C}^*)^6$ to bring $A$ into the standard form
\begin{equation}
\label{eqn-matsumoto_sextic}
\begin{pmatrix}
1 & 0 & 0 & 1 & 1 & 1 \\
0 & 1 & 0 & 1 & x_1 & x_2 \\
0 & 0 & 1 & 1 & x_3 & x_4 
\end{pmatrix}.
\end{equation}
Equivalently, the associated K3 surface $\mathcal{X}$ is the minimal resolution of the double-sextic
\beq
\label{eqn:MATeqn}
 z^2 = t_1 t_2 t_3 \big(t_1 + t_2 + t_3\big) \big(t_1 + x_1 t_2 + x_3 t_3\big)  \big(t_1 + x_2 t_2 + x_4 t_3\big) \,.
\eeq 
In \cite[\S 4]{MR1091686} Sasaki showed that the period integral  for the non-vanishing holomorphic two-form $\eta_{\mathcal{X}} \in H^0(\boldsymbol \omega_{\mathcal{X}})$ induced by $dt_2 \wedge dt_3/z$ in Equation (\ref{eqn:MATeqn}) in the affine chart $t_1=-1$ over a transcendental cycle $\Sigma' \in~\mathrm{T}(\mathcal{X})$, given by 
\begin{equation}
\label{eqn:periods}
\omega'= \omega'(x_1,x_2,x_3,x_4) = \oint_{\Sigma'} \eta_{\mathcal{X}} ,
\end{equation}
is a solution of the resonant rank-six Aomoto-Gel'fand system $E(3,6)$ in the variables $x_1$, $x_2$, $x_3$, $x_4$.  The construction of transcendental cycles $\Sigma'$ was described in  \cite{MR973860}.
\par In the affine coordinate system $t_1=-1$, we consider the transformation $\varphi : \mathbf{X}^{(\mu)} \dashrightarrow \mathcal{X}$ given by 
\begin{equation*}
 t_2 = \frac{x_3 t -1}{x_3 t - x_1}, \hspace{5mm} t_3 = \frac{x(1-x_1)}{x_3 t - x_1}, \hspace{5mm} z=\frac{x_3(x_1-1)^2 \tilde{y}}{(x_3t-x_1)^3},
\end{equation*}
together with the change of parameters in Equation~(\ref{eqn:params}). Here, $\mathbf{X}^{(\mu)}$ is the twist of the K3 surface $\mathbf{X}$ and given by
\begin{equation}
\label{eqn-scaled_yoshida}
\mu \tilde{y}^2 = \, x(x-1)(x-t)(t-a)(t-b)(t-c-dx) 
\end{equation}
with $\mu = b (b-c)$. The map $\varphi : \mathbf{X}^{(\mu)} \dashrightarrow \mathcal{X}$ extends to a birational map of K3 surfaces such that
\begin{equation}
\label{eqn:2forms}
\varphi^* \eta_{\mathcal{X}} = dt \wedge \frac{dx}{\tilde{y}}.
\end{equation}
It follows that periods of the two-form $dt \wedge dx/\tilde{y}$ for the family $\mathbf{X}^{(\mu)}$ satisfy the same Picard-Fuchs system as the periods $\omega'$ in Equation~(\ref{eqn:periods}). In turn, periods $\omega$ of the two-form $dt \wedge dx/y$ for $\mathbf{X}$ in Equation~(\ref{eqn-extended_legendre2}) with $y=\sqrt{\mu} \tilde{y}$ are annihilated by the same Picard-Fuchs operator as $\omega' / \sqrt{\mu}$.
\end{proof}
\par We now turn our attention to the determination of the monodromy group of the period map of the family $\mathbf{X}$ of double sextic $L$-polarized K3 surfaces. As the Picard-Fuchs system $E(3,6)$ annihilating the (twisted) period integral in Proposition \ref{cor:periods} is a resonant GKZ system, the monodromy representation is reducible \cite{MR2966708}, and so the determination of the monodromy group is in general more complicated. Our strategy is to connect the family $\mathbf{X}$ birationally to other families of K3 surfaces whose monodromy groups are known, as we have done in Proposition \ref{cor:periods} with the double sextic family $\mathcal{X}$ studied by Matsumoto et al. \cite{MR1136204}.
\par We need to pay close attention to the twist factor $\sqrt{\mu}=\sqrt{b(b-c)}$, which causes the period map for the family $\mathbf{X}$ to become \emph{multi-valued}; thus, the monodromy representation does \emph{not} coincide with the topological monodromy of the family, i.e., the monodromy of the local system $R^2\pi_*\underline{\mathbb{Z}}_{\mathbf{X}} \to \mathcal{M}$. 
\par Let $\Sigma \in \mathrm{T}(\mathbf{X})$ be a transcendental cycle, and $\nabla$ the Gauss-Manin connection from \S \ref{ss-weierstrass} associated to the system of Picard-Fuchs equations for $\mathbf{X}$ - the Aomoto-Gel'fand $E(3,6)$ system - in Proposition \ref{cor:periods}. Let $\eta_{\mathbf{X}}$ be the holomorphic two-form on the K3 surface $\mathbf{X}$ induced by $dt \wedge dx/y$. As we parallel transport $\Sigma$ under $\nabla$ around the locus $b=0$ in $\mathcal{M}$, for an initial point away from $c=0$, we obtain a new cycle $\Sigma'$ that is related by the action of the monodromy group of the Aomoto-Gel'fand system on $\mathrm{T}(\mathbf{X})$ and the twist $\mu$ relating the families $\mathcal{X}$ and $\mathbf{X}$; see proof of Proposition \ref{cor:periods}. Thus, as we switch branches of the square root of the twisting factor, we obtain the following action on a period integral:
\begin{equation}
\label{eqn:sign_switch}
   \sqrt{b(b-c)} \oint_{\Sigma}\, \eta_{\mathbf{X}} \; \; \;  \to \; \; \; -\sqrt{b(b-c)} \oint_{\Sigma'}\, \eta_{\mathbf{X}} \,.
\end{equation}
\par The situation can be described as follows: let $\Pi \to \mathcal{M}$ be the period sheaf of the family $\mathbf{X}$ described in \S \ref{ss-weierstrass}, that  is the rank six complex local system whose stalks are generated by linearly independent period integrals for $\mathbf{X}$. Moreover, we define a rank one integral local system $\mathsf{S} \to \mathbb{C}^4-Z(\mu)$, with the monodromy group $\mathbb{Z}_2$ around the divisor $\mu=0$. Here, $Z(\mu)$ is the vanishing locus of $\mu$ in $\mathbb{C}^4$. The monodromy representation of the family $\mathbf{X}$ acts on the tensor product $\mathsf{S} \otimes_{\underline{\mathbb{Z}}_\mathcal{M}} \Pi$, with $\mathbb{Z}_2$ acting nontrivially as multiplication by $- \mathbb{I}$, the negative of the identity matrix, as the vanishing locus of $\mu$ is encircled away from the singular locus of the family. Here, we are identifying $\mathsf{S}$ with its restriction to $\mathcal{M}$.
\par Let $p : \mathcal{M} \to \mathcal{P}$ be the period mapping
\begin{equation}
\label{eqn-period_map}
    p: (a,b,c,d) \mapsto [\omega_1(a,b,c,d) \, : \, \cdots \, : \omega_6(a,b,c,d) ] \, , 
\end{equation}
\begin{equation*}
    \omega_i(a,b,c,d) = \oint_{\Sigma_i}\, \eta_{\mathbf{X}} \, , \, \, \, i=1,\dots,6 
\end{equation*}
with $\Sigma_1,\dots,\Sigma_6 \in \mathrm{T}(\mathbf{X})$ a basis, and $\mathcal{P} \subset \mathbb{P}^5$ the period domain of six linearly independent period integrals of the family $\mathbf{X}$ in Equation (\ref{eqn-extended_legendre}). Similarly, for the family $\mathcal{X}$ in Equation (\ref{eqn:MATeqn}) let $\tilde{p}: \mathcal{M}_6 \to \mathcal{P}$ be the period map as defined by Matsumoto \cite[\S 7]{MR1136204}. Let $A$ be the Gram matrix of the lattice $H(2) \oplus H(2) \oplus \langle -2 \rangle^{\oplus 2}$, and let $G_{\mathcal{X}} \subset \mathrm{GL}(6,\mathbb{Z})$ be the subgroup of the isometry group $\mathrm{O}(A,\mathbb{Z}) $ given by 
\begin{equation}
\label{eqn-untwisted_monodromy_group}
G_{\mathcal{X}} = \left\{ M \in \mathrm{GL}(6,\mathbb{Z}) \; | \; M^T A M = A, \; M \equiv \mathbb{I} \! \mod 2 \; \right\} \  \subset\  \mathrm{O}(A,\mathbb{Z})\,.
\end{equation}
We have the following:
\begin{proposition}
\label{cor-monodromy} 
The global monodromy group $G_{\mathbf{X}} \subset \mathrm{GL}(6,\mathbb{Z})$ of the period map $p ~: \mathcal{M} \to \mathcal{P}$ for the family $\mathbf{X}$ in Equation (\ref{eqn-extended_legendre}) is, up to conjugacy, the group $G_{\mathcal{X}}$.
\end{proposition}
\begin{proof}
In \cite[\S 7]{MR1136204}, Matsumoto et al. showed that the monodromy group of the period map $\tilde{p}: \mathcal{M}_6 \to \mathcal{P}$ for the family $\mathcal{X}$ coincides with that of the monodromy group for the Aomoto-Gel'fand $E(3,6)$ system, and is given by the group $G_{\mathcal{X}} \subset \mathrm{O}(A,\mathbb{Z})$ in Equation (\ref{eqn-untwisted_monodromy_group}). They showed this group is the topological monodromy group of $\mathcal{X}$, i.e., the  monodromy group of the local system $R^2\pi_*\underline{\mathbb{Z}}_{\mathcal{X}} \to \mathcal{M}_6$. It then follows from Proposition \ref{cor:periods} that $G_\mathcal{X} \subseteq G_\mathbf{X}$. For $\mu=b(b-c)$, the multi-valued functions $\sqrt{\mu}\, \omega$ were shown to be solutions to Aomoto-Gel'fand $E(3,6)$ system. Hence, the tensor product of local systems $\mathsf{S}\otimes_{\underline{\mathbb{Z}}_\mathcal{M}}\Pi$ is the span of solutions to the Picard-Fuchs system for the family $\mathbf{X}$, where $\mathsf{S}$ is the rank one integral local system defined above.  The order-two monodromy group $\mathbb{Z}_2$ is generated by the monodromy around the vanishing locus of $\mu$, and $\Pi$ is the rank six period sheaf. 
\par Let $\Lambda$  be subset of the parameter space corresponding to singular members of the family $\mathcal{X}$.  Let $g_\gamma$ be the monodromy operator acting on the cohomology of $\mathcal{X}$ for any loop $\gamma$ in $\mathbb{C}^4 \backslash \big(\Lambda \cup Z(\mu)\big)$. The corresponding monodromy operator $h_\gamma$ attached to the same loop applied to the cohomology of $\mathbf{X}$ satisfies $h_\gamma = \pm g_\gamma$ by Equation~(\ref{eqn:sign_switch}). Since $-\mathbb{I} \in G_\mathcal{X}$ it follows that $h_\gamma \in G_\mathcal{X}$ and $G_\mathbf{X} \cdot \{ \pm \mathbb{I} \} = G_\mathcal{X}$. Since $Z(\mu) \not \subset \Lambda$, it follows that $-\mathbb{I} \in G_{\mathbf{X}}$. In fact, for a loop in $\mathcal{M} \cap Z(\mu)$ away from the singular locus of $\mathbf{X}$, the monodromy operator acts nontrivial on the first factor of $\mathsf{S}\otimes_{\underline{\mathbb{Z}}_\mathcal{M}}\Pi$ alone. Hence, we have the equality $G_\mathbf{X}=G_\mathcal{X}$.
\end{proof}
\begin{remark}
The proof of Proposition~\ref{cor-monodromy} shows that the monodromy group of the family $\mathbf{X}$ is  the same as that of $\mathcal{X}$ while the monodromy representations are different. Similar statements hold about the monodromy groups in Corollary \ref{cor-rank_17}, Corollary \ref{cor-rank_18}, and Corollary \ref{cor-rank_19}.
\end{remark}
\subsection{Extensions of the lattice polarization}
\label{ss-degenerations}
Using the four-parameter family of K3 surfaces in Proposition \ref{lemma-yoshida_surface}, we can efficiently study certain extensions of the lattice polarization and identify the corresponding lattice polarizations, monodromy groups, and Picard-Fuchs operators.
\subsubsection{Picard rank $\rho=17$} We consider the extension of the lattice polarization for $d = 0$.  In this case, the surface $\mathbf{X}'_{a,b,c}=\mathbf{X}_{a, b, c, 0}$ becomes the \emph{twisted Legendre Pencil}:
\begin{equation}
\label{eqn-twisted_legendre_pencil}
y^2=x(x-1)(x-t)(t-a)(t-b)(t-c).
\end{equation}
The minimal resolution of a general member has Picard number 17 and was studied by Hoyt~\cite{MR1013162}.  We have the following:
\begin{lemma}
\label{prop-twisted_legendre1}
Equation~(\ref{eqn-twisted_legendre_pencil}) defines a Jacobian  elliptic fibration $\pi : \mathbf{X}' \to ~\mathbb{P}^1$ on a general $\mathbf{X}' = \mathbf{X}'_{a,b,c}$ with three singular fibers of Kodaira type $I_0^*$ over $t=a,b,c$, three singular fibers of Kodaira type $I_2$, and the Mordell Weil group $\mathrm{MW}(\mathbf{X}',\pi) = (\mathbb{Z}/2\mathbb{Z})^2$.  
\end{lemma}
\begin{proof} The proof is similar to the ones given in the preceding section. The statement about Picard rank and the Mordell Weil group can be found in Hoyt \cite{MR1013162}.
\end{proof}
\par In particular, $\mathbf{X}'$ is birational to the two-parameter quadratic twist family of the one parameter family of rational elliptic surfaces $S_{c,d=0}$ from Lemma \ref{prop-ratnet}, and hence, $\mathbf{X}'$ is equivalently described by the mixed-twist construction with generalized functional invariant $(i, j,\alpha)=(1,1,1)$.  We have the following:
\begin{proposition}
\label{thm-double_sextic22}
Over $\mathcal{M}'=\mathcal{M}|_{d=0}$ the family $\mathbf{X}'_{a,b,c}$ in Equation~(\ref{eqn-twisted_legendre_pencil}) is a 3-dimensional family of $L'$-polarized K3 surfaces $\mathbf{X}'$ where $L'$ has rank 17 and the following isomorphic presentations:
\begin{equation}
\label{eqn:Lp}
\begin{split}
 L'  \cong \ & H \oplus E_8(-1) \oplus D_4(-1) \oplus A_1(-1)^{\oplus 3} \ \cong \ H \oplus E_7(-1)  \oplus D_4(-1)^{\oplus 2} \\
\ \cong \ & H \oplus D_{12}(-1) \oplus A_1(-1)^{\oplus 3}  \ \cong \ H \oplus D_{10}(-1) \oplus D_4(-1) \oplus A_1(-1) \\
\cong \ &  H \oplus D_8(-1) \oplus D_6(-1) \oplus A_1(-1) .
\end{split}
\end{equation}
In particular, $L'$ is a primitive sub-lattice of the K3 lattice $\Lambda_{K3}$.
\end{proposition}
\begin{proof}
We use the same strategy as in the proof of Proposition~\ref{cor:periods}. Using Lemma~\ref{prop-twisted_legendre1} it follows that the two-elementary lattice $L'$ must have $\rho=17$ and $\ell=5$. Applying Nikulin's classification \cite{MR633160} it follows that there is only one such lattice admitting a primitive lattice embedding into $\Lambda_{K3}$, and it must have $\delta=1$. We then go through the list in \cite{MR1813537} to find the isomorphic presentations.
\end{proof}
\begin{remark}
In \cite{CM:2017} it was shown that the configuration of six lines $\boldsymbol \ell$ associated with $\mathbf{X}'$ has three lines intersecting in one point.  The pencil of lines through the intersection point induces precisely the elliptic fibration of Lemma~\ref{prop-twisted_legendre1}. In particular, the general K3 surface $\mathbf{X}'$ is not a Jacobian Kummer surface. It is the relative Jacobian fibration of an elliptic Kummer surface associated to an abelian surface with a polarization of type $(1,2)$; this was proved in \cite{CM:2017, CMS:2020}.
\end{remark}
\par Setting $d=0$ in Proposition~\ref{cor:periods} we immediately obtain the following:
\begin{corollary}
\label{prop-twisted_legendre2}
Let $\Sigma \in \mathrm{T}(\mathbf{X}')$ be a transcendental cycle on a general K3 surface $\mathbf{X}' = \mathbf{X}'_{a,b,c}$, $\eta_{\mathbf{X}'}$ the holomorphic two-form induced by $dt \wedge dx/y$ in Equation~(\ref{eqn-twisted_legendre_pencil}), and $\omega = \oint_{\Sigma} \eta_{\mathbf{X}'}$ a period. The Picard-Fuchs system for $\mathbf{X}'_{a,b,c}$, annihilating $\omega'= \sqrt{b(b-c)} \, \omega$, is the restricted rank-five Aomoto-Gel'fand system $E(3,6)$ of \cite{MR973860, MR1073363, MR1136204} with $x_4=0$. 
\end{corollary}
To determine the global monodromy group of the period map for the twisted Legendre pencil, we utilize the relation of $\mathbf{X}'$ to the Kummer surface $\mathrm{Kum}(A)$ of a principally polarized abelian surface $A$. This is equivalent to determining which configurations of six lines $\boldsymbol \ell$ yield total spaces that are Kummer surfaces; in particular, the lines must be mutually tangent to a common conic. In \cite{MR1953527} the authors gave geometric characterizations of such six-line configurations. We have the following:
\begin{proposition}
\label{prop:KummerJac}
The minimal resolution of a general member in Equation~(\ref{eqn-extended_legendre}) is a Jacobian Kummer surface, i.e., the Kummer surface associated with the Jacobian of a general genus-two curve, if and only if  $d(ab-b)= (a-c)(b-c)$.
\end{proposition}
\begin{proof}
Using the methods of \cite{MR4015343} we compute the square of the degree-two Dolgachev-Ortland invariant $R^2$. It vanishes if and only if the six lines are tangent to a common conic. It is well known that this is a necessary and sufficient criterion for the total space to be a Jacobian Kummer surface; see for example \cite{Clingher:2018aa}. A direct computation of $R^2$ for the six lines in Equation~(\ref{eqn-extended_legendre}) yields the result.
\end{proof}
We also have the following:
\begin{lemma}
For general parameters $a, b, c$ and $d= (a-c)(b-c)/(ab-c)$ Equation~(\ref{eqn-extended_legendre}) defines a Jacobian  elliptic fibration $\pi : \widetilde{\mathbf{X}} \to ~\mathbb{P}^1$ with the singular fibers $2 I_0^* + 6 I_2$ and the Mordell Weil group $ \mathrm{MW}(\widetilde{\mathbf{X}},\pi) =(\mathbb{Z}/2\mathbb{Z})^2 \oplus \langle 1 \rangle$.
\end{lemma}
The connection between the parameters $a, b, c$ and  the moduli of genus-two curves was exploited in \cite{MR3731039, MR3782461}. We have the following:
\begin{proposition}
\label{thm-double_sextic2alt}
Over the subspace $\widetilde{\mathcal{M}}$, given as $d= (a-c)(b-c)/(ab-c)$ in $\mathcal{M}$, the family in Equation~(\ref{eqn-rational}) is a three-dimensional family of $\tilde{L}$-polarized K3 surfaces $\widetilde{\mathbf{X}}$ where $\tilde{L}$ has the following isomorphic presentations:
\begin{equation}
\label{eqn:polarization}
 \tilde{L}  \ \cong \  H \oplus D_8(-1) \oplus D_4(-1) \oplus A_3(-1) \ \cong \ H \oplus D_7(-1) \oplus D_4(-1)^{\oplus 2} \,.
\end{equation}
In particular, $\tilde{L}$ is a primitive sub-lattice of the K3 lattice $\Lambda_{K3}$.
\end{proposition}
\begin{proof}
We established in Proposition~\ref{prop:KummerJac} that the K3 surface obtained from the Weierstrass model in Equation~(\ref{eqn-extended_legendre}) is a Jacobian Kummer surface if and only if the parameters $a, b, c, d$ satisfy a certain relation. In \cite{MR3263663} Kumar classified all Jacobian elliptic fibrations on a generic Kummer surface. Among them are exactly two fibrations that have a trivial Mordell Weil group, called (15) and (17). The types of reducible fibers in the two fibrations then yield isomorphic presentations for the polarizing lattice.
\end{proof}
\begin{remark}
It was shown in \cite{CM:2017} that the general K3 surface $\widetilde{\mathbf{X}}$ in Proposition~\ref{thm-double_sextic2alt} arises as the rational double cover of a general K3 surface in Proposition~\ref{prop-twisted_legendre1}. The double cover $\widetilde{\mathbf{X}} \dasharrow \mathbf{X}'$ is branched along the even eight on $\mathbf{X}'$ composed of the non-central components of the two reducible fibers of type $\widetilde{D}_4$.
\end{remark}
We now determine the monodromy group for the period map of the twisted Legendre pencil $\mathbf{X}'$ in Equation (\ref{eqn-twisted_legendre_pencil}). Notice that the period map for this family is the restriction $p\!\mid_{\mathcal{M}'}$ to $\mathcal{M}'$ of the period map from Equation (\ref{eqn-period_map}). We define a rank-one integral local system $\mathsf{S}' \to \mathbb{C}^3-Z(\mu)$, by restricting the local system $\mathsf{S}$ defined above as $\mathsf{S}'=\mathsf{S}\vert_{d=0}$. The monodromy around the locus $\mu=0$ obtained by switching branches of the square root function and is again $\mathbb{Z}_2$.
\par In the following, for a matrix group $G \subseteq \mathrm{GL}(n,\mathbb{Z})$, identified with its standard representation acting on $\mathbb{Z}^n$, let $\wedge^2 \, G \subseteq \mathrm{GL}(r,\mathbb{Z})$ be the exterior square representation acting on $\mathbb{Z}^r$, with $r=\binom{n}{2}$. In the following result, the exterior square representation of the group $G$ turns out to be reducible on $\mathbb{Z}^r$, but irreducible on $\mathbb{Z}^{r-1}$. Let $\Gamma_2(2) \subset~ \mathrm{Sp}(4,\mathbb{Z})$ be the Siegel congruence subgroup of level two. Hara et al. showed in \cite{MR1040172} that the exterior square representation $\wedge^2 \, \Gamma_2(2) \subset \mathrm{GL}(6,\mathbb{Z})$ of the Siegel congruence subgroup of level two $\Gamma_2(2) \subset \mathrm{Sp}(4,\mathbb{Z})$ is reducible on $\mathbb{Z}^6$, but irreducible on $\mathbb{Z}^5$. Hence, we have $\wedge^2 \, \Gamma_2(2) \subset \mathrm{GL}(5,\mathbb{Z})$. 
\begin{corollary}
\label{cor-rank_17}
The global monodromy group $G_{\mathbf{X}'} \subset \mathrm{GL}(5,\mathbb{Z})$ of the period map $p\!\mid_{\mathcal{M}'}$ is, up to conjugacy, the exterior square $G_{\mathbf{X}'}=\wedge^2\, \Gamma_2(2)$. 
\end{corollary}
\begin{proof}
The period map $p\!\mid_{\mathcal{M}'}$ of the twisted Legendre pencil in Equation (\ref{eqn-twisted_legendre_pencil}) was originally investigated by Hoyt in \cite{MR1013162}, where a partial analysis of its behavior for generic parameter values $a,b,c$ was made. There, Hoyt showed \cite[\S 5, statements (iv$'$), (iv$''$)]{MR1013162} that $\mathbf{X}'$ was related the Kummer surface $\widetilde{\mathbf{X}}=\mathrm{Kum}(\mathrm{Jac}(C))$ of a Jacobian of a general genus-two curve $C$. In Braeger et al. \cite[Theorem 3.12]{MR4099481}, the authors produced a dominant rational rational map $\psi: \widetilde{\mathbf{X}} \dashrightarrow \mathbf{X}'$ of degree two that explicitly related the twisted Legendre parameters $a,b,c$ to the Rosenhain roots $\lambda_1,\lambda_2,\lambda_3$ of the genus two curve $C$ that pulls back the holomorphic two-form $\eta_{\mathbf{X}'}=dt \wedge dx/y$ to the holomorphic two-form $\eta_{\widetilde{\mathbf{X}}}$ on the Kummer surface $\widetilde{\mathbf{X}}$. In particular, the induced map on homology $\psi_* : H_2(\widetilde{\mathbf{X}},\mathbb{C}) \to   H_2\left(\mathbf{X}',\mathbb{C}\right)$ is compatible with the associated lattice polarizations $L'$ on $\mathbf{X}'$ and $\tilde{L}$ on $\widetilde{\mathbf{X}}$. Thus, the Picard-Fuchs systems for $\mathbf{X}'$ and $\widetilde{\mathbf{X}}$ are equivalent. Hara et al. showed in \cite{MR1040172} showed that the global monodromy group of the Picard-Fuchs system for $\widetilde{\mathbf{X}}$ is precisely this exterior square representation $\wedge^2 \, \Gamma_2(2)$. Hence, we have that $\wedge^2 \,\Gamma_2(2) \subseteq G_{\mathbf{X}'}$. Let $\Pi' \to \mathcal{M}'$ is the rank five period sheaf of the family $\mathbf{X}'$, and $\mathsf{S}'$ the rank one integral local system defined above. Then the argument in Proposition \ref{cor-monodromy} applies to the tensor product $\mathsf{S}' \otimes_{\underline{\mathbb{Z}}_{\mathcal{M}'}}\Pi'$ generated by solutions to the Picard-Fuchs equations in Corollary~ \ref{prop-twisted_legendre2}, and it follows that the full monodromy group is $G_{\mathbf{X}'}=\wedge^2\, \Gamma_2(2)$, as desired.
\end{proof}
\subsubsection{Picard rank $\rho=18$} We consider the extension of the lattice polarization for $c=d = 0$.    In this case, the surface $\mathbf{X}''_{a, b}=\mathbf{X}_{a, b, 0,0}$ becomes the two-parameter twisted Legendre pencil:
\begin{equation}
\label{eqn-twisted_legendre_pencil2}
y^2=x(x-1)(x-t)t(t-a)(t-b).
\end{equation}
The minimal resolution of a general member of this family has Picard number 18. We have the following:
\begin{lemma}
\label{prop-two_param_a}
Equation~(\ref{eqn-twisted_legendre_pencil2}) defines a Jacobian elliptic fibration $\pi : \mathbf{X}'' \to ~\mathbb{P}^1$ on a general $\mathbf{X}'' = \mathbf{X}''_{a, b}$ with the singular fibers $I_2^* + 2 I_0^* + 2 I_2$ and the Mordell Weil group $\mathrm{MW}(\mathbf{X}'',\pi) = (\mathbb{Z}/2\mathbb{Z})^2$. 
\end{lemma}
We then have the following:
\begin{proposition}
\label{thm-two_param_b}
Over $\mathcal{M}''=\mathcal{M}|_{c=d=0}$ the family $\mathbf{X}''_{a,b}$ in Equation~(\ref{eqn-twisted_legendre_pencil3}) is a 2-dimensional family of $L''$-polarized K3 surfaces $\mathbf{X}''$ where $L''$ has rank 18 and the following isomorphic presentations:
\begin{equation}
\label{eqn:Lpp}
\begin{split}
 L''  \cong \ & H \oplus E_8(-1) \oplus D_6(-1) \oplus A_1(-1)^{\oplus 2} \ \cong \ H \oplus E_7(-1)^{\oplus 2}   \oplus A_1(-1)^{\oplus 2} \\
\ \cong \ & H \oplus E_7(-1) \oplus D_8(-1) \oplus A_1(-1) \ \cong \ H \oplus D_{14}(-1) \oplus A_1(-1)^{\oplus 2} \\
\cong \ & H \oplus D_{10}(-1) \oplus D_6(-1) .
\end{split}
\end{equation}
In particular, $L''$ is a primitive sub-lattice of the K3 lattice $\Lambda_{K3}$.
\end{proposition}
\begin{proof}
We use the same strategy as in the proof of Proposition~\ref{cor:periods}.  Using Lemma~\ref{prop-two_param_a} it follows that the two-elementary lattice $L''$ must have $\rho=18$ and $\ell=4$. Applying Nikulin's classification \cite{MR633160} it follows that there are two such lattices admitting a primitive lattice embedding into $\Lambda_{K3}$, namely the ones with $\delta=0, 1$. A standard lattice computation shows that we have $\delta=1$. We then go through the list in \cite{MR1813537} to find the isomorphic presentations.
\end{proof}
From \cite[Corollary 2.2]{MR3767270}, the Picard-Fuchs system can now be determined explicitly:
\begin{proposition}
\label{prop-two_param}
Let $\Sigma \in \mathrm{T}(\mathbf{X}'')$ be a transcendental cycle on a general K3 surface $\mathbf{X}''$, $\eta_{\mathbf{X}''}$ the holomorphic two-form induced by $dt \wedge dx/y$ in Equation~(\ref{eqn-twisted_legendre_pencil2}), and $\omega = \oint_{\Sigma} \eta_{\mathbf{X}''}$ a period. The Picard-Fuchs system for $\mathbf{X}''_{a,b}$, annihilating $\omega'=\omega/\sqrt{a} $, is the Appell's rank four hypergeometric system $F_2 (\frac{1}{2},\frac{1}{2},\frac{1}{2};1,1 | 1 - \frac{b}{a},  b )$.
\end{proposition}
\begin{proof}
We consider the transformation $\varphi : \mathbf{X}^{(\mu)} \dashrightarrow \mathbf{X}''$ given by 
\begin{equation*}
    t=\frac{ab}{a+(b-a)T} \; , \hspace{5mm} x= \frac{1}{X} \; , \hspace{5mm} y = \frac{ab(b-a) \tilde{Y}}{(a+(b-a)T)^2 X^2} \, .
\end{equation*}
Here, $\mathbf{X}^{(\mu)}$ is the twisted Legendre pencil
\begin{equation}
    \mu\, \tilde{Y}^2=X(1-X) T (1-T) (1- a' T - b X)\, ,
\end{equation}
with $a'=1-b/a$ and $\mu=(1-a')/b =1/a$. The map $\varphi : \mathbf{X}^{(\mu)} \dashrightarrow \mathbf{X}''$ induces a birational equivalence extending to a birational map of K3 surfaces such that
\begin{equation}
\label{eqn:2formsbb}
\varphi^* \frac{dt \wedge dx}{y} =  \frac{dT \wedge dX}{\tilde{Y}} .
\end{equation}
It is known that periods $\omega'$ of the two-form $dT \wedge dX/Y$ for the (untwisted) family with $Y = \sqrt{\mu} \, \tilde{Y}$ and
\begin{equation}
   Y^2=X(1-X) T (1-T) (a' T + b X - 1)\, 
\end{equation}
satisfy the Appell's hypergeometric system of $F_2(\frac{1}{2},\frac{1}{2},\frac{1}{2};1,1 | a', b)$. Thus, periods $\omega$ of  $dt \wedge dx/y$ for $\mathbf{X}''$ satisfy the same differential system as $\omega' / \sqrt{\mu}$.
\end{proof}
We now determine the monodromy group for the period map for the family $\mathbf{X}''$ in Equation (\ref{eqn-twisted_legendre_pencil2}). In this case, the period map coincides with the restriction of the period map $p\!\mid_{\mathcal{M}''}$ from Equation (\ref{eqn-period_map}). Again, we introduce a rank-one integral local system $\mathsf{S}''\to \mathbb{C}^2 - Z(1/\mu)$, with $\mu=1/a$, to record the monodromy around the locus $\mu=0$ obtained by switching branches of the square root function.
\par For a matrix group $G \subseteq \mathrm{GL}(n,\mathbb{Z})$, identified with its standard representation acting on $\mathbb{Z}^n$, let $G \boxtimes G \subseteq \mathrm{GL}(2n,\mathbb{Z})$ be the outer tensor product representation of $G$ acting on $\mathbb{Z}^{2n}$. Let $\Gamma(2) \subset \mathrm{SL}(2,\mathbb{Z})$ the principal congruence subgroup of level two.
\begin{corollary}
\label{cor-rank_18}
The global monodromy group $G_{\mathbf{X}''}$ of the period map $p\!\mid_{\mathcal{M}''}$ is, up to conjugacy, the outer tensor product $G_{\mathbf{X}''}=\Gamma(2) \boxtimes \Gamma(2)$.
\end{corollary}
\begin{proof}
In \cite[Theorem 2.5]{MR3767270}, Clingher et al.~showed that the period integral of the twisted Legendre pencil in Equation (\ref{eqn-twisted_legendre_pencil2}) of Picard rank $\rho \geq 18$ factorizes holomorphically into two copies of the Gauss hypergeometric function $\hpgo{2}{1}(\frac{1}{2},\frac{1}{2},1 \, | \; \cdot \; )$. At the level of Picard-Fuchs systems, this is realized as the decoupling of the rank four Fuchsian system annihilating Appell's $F_2$ function from Proposition \ref{prop-two_param} into two copies of the rank two Fuchsian ODE annihilating $\hpgo{2}{1}(\frac{1}{2},\frac{1}{2},1 \, | \; \cdot \; )$. The monodromy group of each ODE is known to be the principal congruence subgroup of level two $\Gamma(2) \subset~ \mathrm{SL}(2,\mathbb{Z})$. It follows that $\Gamma(2) \boxtimes \Gamma(2) \subseteq G_{\mathbf{X}''}$. Let $\Pi'' \to \mathcal{M}''$ be the rank four period sheaf of the family $\mathbf{X}''$, and $\mathsf{S}''$ the rank one integral local system defined above. We apply the argument from the proof of Proposition \ref{cor-monodromy} to the tensor product $\mathsf{S}''\otimes_{\underline{\mathbb{Z}}_{\mathcal{M}''}}\Pi''$ generated by solutions to the Picard-Fuchs equations in Proposition \ref{prop-two_param}, and obtain the full monodromy group $G_{\mathbf{X}''}=\Gamma(2) \boxtimes \Gamma(2)$, as desired.
\end{proof}
\subsubsection{Picard rank $\rho=19$}
We consider the extension of the lattice polarization for $c=d = 0$ and $b=1$.  In this case, the surface $\mathbf{X}'''_{a}=\mathbf{X}_{a,1, 0, 0}$ becomes the one-parameter twisted Legendre pencil:
\begin{equation}
\label{eqn-twisted_legendre_pencil3}
y^2=x(x-1)(x-t)t(t-1)(t-a).
\end{equation}
This family was studied in detail by Hoyt \cite{MR894512}; the general member has Picard number $\rho=19$. We have the following:
\begin{lemma}
\label{prop-one_param_a}
Equation~(\ref{eqn-twisted_legendre_pencil3}) defines a Jacobian elliptic fibration $\pi : \mathbf{X}''' \to ~\mathbb{P}^1$ on a general $\mathbf{X}''' = \mathbf{X}'''_{a}$  with the singular fibers $2 I_2^* + I_0^* + 2 I_2$ and the Mordell Weil group $\mathrm{MW}(\mathbf{X}''',\pi) = (\mathbb{Z}/2\mathbb{Z})^2$. 
\end{lemma}
We then have the following:
\begin{proposition}
\label{thm-one_param_b}
Over $\mathcal{M}'''=\mathcal{M}|_{b=1, c=d=0}$  the family $\mathbf{X}'''_{a}$ in Equation~(\ref{eqn-twisted_legendre_pencil3}) is a 1-dimensional family of $L'''$-polarized K3 surfaces $\mathbf{X}'''$ where $L'''$ has rank 19 and the following isomorphic presentations:
\begin{equation}
\label{eqn:Lppp}
\begin{split}
 L'''  \ \cong \ & H \oplus E_8(-1) \oplus E_7(-1) \oplus A_1(-1)^{\oplus 2} \ \cong \ H \oplus E_7(-1) \oplus D_{10}(-1) \\
\cong \ & H \oplus E_8(-1) \oplus D_8(-1) \oplus A_1(-1)  \ \cong \ H \oplus D_{16}(-1)  \oplus A_1(-1)  .
\end{split}
\end{equation}
In particular, $L'''$ is a primitive sub-lattice of the K3 lattice $\Lambda_{K3}$.
\end{proposition}
\begin{proof}
We use the same strategy as in the proof of Proposition~\ref{cor:periods}.  Using Lemma~\ref{prop-one_param_a} it follows that the two-elementary lattice $L'''$ must have $\rho=19$ and $\ell=3$. Applying Nikulin's classification \cite{MR633160} it follows that there is only one such lattice admitting a primitive lattice embedding into $\Lambda_{K3}$, and it must have $\delta=1$. We then go through the list in \cite{MR1813537} to find the isomorphic presentations.
\end{proof}
We have the following:
\begin{proposition}
\label{prop-one-param}
Let $\Sigma \in \mathrm{T}(\mathbf{X}''')$ be a transcendental cycle on a general K3 surface $\mathbf{X}'''=\mathbf{X}'''_{a}$, $\eta_{\mathbf{X}'''}$ the holomorphic two-form induced by $dt \wedge dx/y$ in Equation~(\ref{eqn-twisted_legendre_pencil3}), and $\omega = \oint_{\Sigma} \eta_{\mathbf{X}'''}$ a period. The Picard-Fuchs operator for $\mathbf{X}'''_{a}$, annihilating $\omega' = \omega/\sqrt{a} $, is the rank three ordinary differential operator annihilating the generalized hypergeometric function $\hpgo{3}{2}(\frac{1}{2},\frac{1}{2},\frac{1}{2};1,1 | 1 - \frac{1}{a})$.\end{proposition}
\begin{remark}
The results in Propositions~\ref{prop-one-param} and~\ref{prop-two_param} are in agreement with \cite[Thm.~2.1]{MR2514458} and \cite{MR3767270} where it was shown that the two restrictions
\begin{equation}
\label{2IndepSolns}
 \hpg32{\alpha,\,\beta_1, \, 1+\alpha-\gamma_2}{\gamma_1, \, 1+ \alpha-\gamma_2 + \beta_2}{z_1} \quad \text{and} \quad \app2{\alpha;\;\beta_1,\beta_2}{\gamma_1,\gamma_2}{z_1,1} 
\end{equation}
satisfy the same ordinary differential equation.
\end{remark}
\begin{proof}
We consider the transformation $\varphi : \mathbf{X}^{(\mu)} \dashrightarrow \mathbf{X}'''$ given by 
\begin{equation*}
    t=\frac{a}{a+(1-a)T} \; , \hspace{5mm} x= -\frac{a(1-X)}{(a+(1-a)T)X} \; , \hspace{5mm} y = -\frac{(1-a)a^2 \tilde{Y}}{(a+(1-a)T)^3 X^2} \, ,
\end{equation*}
Here, $\mathbf{X}^{(\mu)}$ is the twisted Legendre pencil
\begin{equation}
    \mu\, \tilde{Y}^2=X(1-X) T (1-T) (1- a' TX)\, ,
\end{equation}
with $a'=1-1/a$ and $\mu=1-a'$. The map $\varphi : \mathbf{X}^{(\mu)} \dashrightarrow \mathbf{X}'''$ induces a birational equivalence extending to a birational map of K3 surfaces such that
\begin{equation}
\label{eqn:2formsb}
\varphi^* \frac{dt \wedge dx}{y} =  \frac{dT \wedge dX}{\tilde{Y}} .
\end{equation}
It is known that periods $\omega'$ of the two-form $dT \wedge dX/Y$ for the (untwisted) family with $Y = \sqrt{\mu} \, \tilde{Y}$ and
\begin{equation}
   Y^2=X(1-X) T (1-T) (1- a' TX)\, 
\end{equation}
satisfy the differential equation of $\hpgo{3}{2}(\frac{1}{2},\frac{1}{2},\frac{1}{2};1,1 | a')$. Thus, periods $\omega$ of  $dt \wedge dx/y$ for $\mathbf{X}'''$ satisfy the same differential equation as $\omega' / \sqrt{\mu}$.
\end{proof}
\par We determine the monodromy group for the period map for the family $\mathbf{X}'''$ in Equation (\ref{eqn-twisted_legendre_pencil3}). The period map coincides with the restriction of the period map $p\!\mid_{\mathcal{M}'''}$ from Equation (\ref{eqn-period_map}). We again define here a rank-one integral local system $\mathsf{S}'''\to \mathbb{C} - Z(1/\mu)$ by restricting the local system $\mathsf{S}''$ in Corollary \ref{cor-rank_18} and the preceding discussion there as $\mathsf{S}'''=\mathsf{S}''|_{b=1}$, as to record the monodromy around the locus $1/\mu=0$ obtained by switching branches of the square root $\sqrt{\mu}$ with $\mu=1/a$. 
\par In the following, for a matrix group $G \subseteq \mathrm{GL}(n,\mathbb{Q})$, identified with its standard representation acting on $\mathbb{Q}^n$, let $G \odot G \subseteq \mathrm{GL}(r,\mathbb{Q})$ be the symmetric square representation acting on $\mathbb{Z}^r$, with $r=n(n+1)/2$. We also denote by $\Gamma(2)^* := \langle \Gamma(2), w \rangle$ with $ w = \big(\begin{smallmatrix}
  0 & -\frac12\\
  2 & 0
\end{smallmatrix}\big)$ the Fricke involution.
\begin{corollary}
\label{cor-rank_19}
The global monodromy group $G_{X'''} \subset \mathrm{GL}(3,\mathbb{Z})$ of the period map $p\!\mid_{\mathcal{M}'''}$ is, up to conjugacy, the direct product  $G_{\mathbf{X}'''}= \Gamma(2)^* \odot \Gamma(2)^*$.  
\end{corollary}
\begin{proof}
Equation~(\ref{eqn:2formsb}) proves that the monodromy group of the ODE annihilating $\hpgo{3}{2}(\frac{1}{2},\frac{1}{2},\frac{1}{2};1,1 | \; \cdot \; )$  is the symmetric square representation in $\mathrm{GL}(3,\mathbb{Z})$ of the monodromy group for the ODE annihilating $\hpgo{2}{1}(\frac{1}{2},\frac{1}{2};1 | \; \cdot \; )$, after adjoining the involution that is generated by the monodromy operator for loops around the singular fiber at $t=a$ or, equivalently, $t=0$. One checks that in terms of the modular parameter the action is conjugate to the action of the Fricke involution $w$. Hence, we have $\Gamma(2)^* \odot \Gamma(2)^* \subseteq~ G_{\mathbf{X}'''}$. Let $\Pi''' \to \mathcal{M}'''$ be the rank three period sheaf of the family $\mathbf{X}'''$, and $\mathsf{S}'''$ the rank one integral local system defined above. Applying the argument from the proof of Proposition \ref{cor-monodromy} to the tensor product $\mathsf{S}''' \otimes_{\underline{\mathbb{Z}}_{\mathcal{M}'''}}\Pi'''$ generated by solutions to the Picard-Fuchs equations in Proposition \ref{prop-one-param}, we obtain the full monodromy group $G_{\mathbf{X}'''}=\Gamma(2)^* \odot \Gamma(2)^*$, as ~desired.
\end{proof}
In general, if $L \leqslant  L' \leqslant \Lambda_{K3}$ are lattices primitively embedded in the K3 lattice, then there is a map $\mathcal{M}_{L' } \to  \mathcal{M}_L$ of moduli spaces which depends on the particular choice of the lattice embeddings. In particular, the map may have degree greater than one.  We have constructed a family of K3 surfaces $\mathbf{X}_{a, b, c, d}$ such that the period map (from the base of the family) to the coarse moduli space $\mathcal{M}_L$ of $L$-polarized K3 surfaces is birational.  We then showed that the restriction of the Weierstrass model for  $\mathbf{X}_{a, b, c, d}$ to a suitable subspace $\mathcal{M}' \subset \mathcal{M}$ with $\dim \mathcal{M}' = \dim \mathcal{M}_{L'}$ determines an extension of the lattice polarization $L' = H \oplus K'$ of $L = H \oplus K$ as extension of the associated root lattices $K^{\text{root}} \hookrightarrow (K')^{\text{root}}$ in the Weierstrass model.  We have the following main result:
\begin{theorem}
\label{theorem1}
Over the subspaces, obtained by restriction and given by
\beq
 \mathcal{M} \ \supset \ \mathcal{M}' = \mathcal{M} \Big|_{d=0}  \ \supset \ \mathcal{M}'' = \mathcal{M}\Big|_{c=d=0}   \ \supset \ \mathcal{M}''' = \mathcal{M}\Big|_{b=1,c=d=0} \,,
\eeq
the polarization of the family $\mathbf{X}_{a, b, c, d}$ extends in a chain of even, indefinite, two-elementary lattices, given by
\begin{equation}
 L \ \leqslant  \ L' \ \leqslant  \ L'' \  \leqslant  \ L''' \,,
\end{equation}
where the lattices are uniquely determined by (rank, length, parity) with $(\rho, \ell, \delta) = (16+k, 6-k, 1)$ for $k=0, 1, 2, 3$ such that $\dim \mathcal{M}^{(k)} = \dim \mathcal{M}_{L^{(k)}} = 4-k$.  Their Picard-Fuchs systems are determined in Proposition~\ref{cor:periods}, Corollary \ref{prop-twisted_legendre2}, and Propositions \ref{prop-two_param}, \ref{prop-one-param},\ and the global monodromy groups in Proposition~\ref{cor-monodromy}, and Corollaries \ref{cor-rank_17}, \ref{cor-rank_18}, \ref{cor-rank_19}.
\end{theorem}
\begin{proof}
Restricting (i) $d=0$, (ii) $c=d=0$, (iii) $b=1, c=d=0$ in the family of K3 surfaces in Equation~(\ref{eqn-extended_legendre}), the theorem collect statements from Propositions~\ref{lemma-yoshida_surface}, \ref{thm-double_sextic22}, \ref{thm-two_param_b}, \ref{thm-one_param_b} and their respective proofs, as well as from Proposition~\ref{cor:periods}, Corollary \ref{prop-twisted_legendre2}, Propositions \ref{prop-two_param}, \ref{prop-one-param} and Proposition~\ref{cor-monodromy}, Corollaries \ref{cor-rank_17}, \ref{cor-rank_18}, \ref{cor-rank_19}.
\end{proof}
\section{GKZ Description of the Univariate Mirror Families}
\label{GKZ}
In this section we will show that the generalized functional invariant of the mixed-twist construction captures all key features of the one-parameter mirror families for the Fermat pencils. In particular, we will show that the mixed-twist construction allows us to obtain a non-resonant GKZ system for which a basis of solutions in the form of absolutely convergent Mellin-Barnes integrals exists whose monodromy is computed explicitly.
\subsection{The Mirror Families} 
\label{ss-mirror}
\par Let us briefly review the construction of the mirror family for the deformed Fermat hypersurface. Let $\mathbb{P}^n(n+1)$ be the general family of hypersurfaces of degree $(n+1)$ in $\mathbb{P}^n$. The general member of $\mathbb{P}^n(n+1)$ is a smooth hypersurface Calabi-Yau $(n-1)$-fold. Let $[X_0 : \cdots : X_n]$ be the homogeneous coordinates on $\mathbb{P}^n$. The following family 
\begin{equation}
\label{eqn-deformed_fermat}
X_0^{n+1} + \cdots + X_n^{n+1} + n\lambda X_0 X_1 \cdots X_n = 0
\end{equation}
determines a one-parameter single-monomial deformation $X^{(n-1)}_{\lambda}$ of the classical Fermat hypersurface in $\mathbb{P}^{n}(n+1)$. Cox and Katz determined \cite{MR1677117} what deformations of Calabi-Yau hypersurfaces remain Calabi-Yau. For example, for $n=5$ there are 101 parameters for the complex structure, which determine  the coefficients of additional terms in the quintic polynomials. Starting with a Fermat-type hypersurfaces $V$ in $\mathbb{P}^{n}$, Yui~\cite{MR1860046, MR1744960, MR1023927} and Goto~\cite{MR1399698} classified all discrete symmetries $G$ such that the quotients $V/G$ are singular Calabi-Yau varieties with at worst Abelian quotient singularities.  A theorem by Greene, Roan,  and Yau~\cite{MR1137063} guarantees that there are crepant resolutions of $V/G$. This is known as the \emph{Greene-Plesser orbifolding} construction.
\par For the family~(\ref{eqn-deformed_fermat}), the discrete group of symmetries needed for the Greene-Plesser orbifolding is readily constructed: it is generated by the action $(X_0,X_j) \mapsto (\zeta_{n+1}^n X_0, \zeta_{n+1} X_j)$ for $1 \le j \le n$ and the root of unity $\zeta_{n+1}=\exp{(\frac{2\pi i}{n+1})}$. Since the product of all generators multiplies the homogeneous coordinates by a common phase, the symmetry group is $G_{n-1}=(\mathbb{Z}/(n+1)\, \mathbb{Z})^{n-1}$. One checks that the affine variables
\begin{equation*}
 t=\frac{(-1)^{n+1}}{\lambda^{n+1}}, \,  x_1 = \frac{X_1^n}{(n+1)\, X_0 \cdot X_2 \cdots X_n \, \lambda}, \, \dots , \,  x_n
 =\frac{X_2^n}{(n+1)\, X_0 \cdot X_1  \cdots X_{n-1} \, \lambda}, 
 \end{equation*}
are invariant under the action of $G_{n-1}$, hence coordinates on the quotient $X^{(n-1)}_{\lambda}/G_{n-1}$. A family of  special hypersurfaces $Y^{(n-1)}_{t}$ is then defined by the remaining relation between $x_1, \dots, x_n$, namely the equation
\begin{equation}
\label{mirror_family}
 f_n(x_1,\dots, x_n, t) = x_1 \cdots x_n \, \Big( x_1 + \dots + x_n + 1 \Big) +  \frac{(-1)^{n+1} \, t}{(n+1)^{n+1}} =0 \;.
\end{equation}
Moreover, it was proved by Batyrev and Borisov in ~\cite{MR1416334} that the family of special Calabi-Yau hypersurfaces $Y^{(n-1)}_{t}$ of degree $(n+1)$ in $\mathbb{P}^n$ given by Equation~(\ref{mirror_family}) is \emph{mirror} to a general hypersurface $\mathbb{P}^n(n+1)$ of degree $(n+1)$ and co-dimension one in $\mathbb{P}^n$, in the sense that the Hodge diamonds are mirror images, $h^{i,j}(X^{n-1}_{\lambda})=h^{j,i}(Y^{n-1}_t)$ for all $n\geq 2$ and appropriate $\lambda,t$. For $n= 2, 3, 4$ the mirror family is a family of elliptic curves, K3 surfaces, and Calabi-Yau threefolds, respectively.
\par Each mirror family can be realized as a fibration of Calabi-Yau $(n-2)$-folds associated with a generalized functional invariant. The following was proved by Doran and Malmendier:
\begin{proposition}
\label{MirrorRecursive}
 For $n\ge 2$ the family of hypersurfaces $Y^{(n-1)}_{t}$ in Equation~(\ref{mirror_family}) is a fibration over $\mathbb{P}^1$ by hypersurfaces $Y^{(n-2)}_{\tilde{t}}$ constructed as mixed-twist with the generalized functional invariant $(1,n,1)$.
\end{proposition}
\begin{proof}
For each $x_n \not= 0, -1$  substituting $\tilde x_i= x_i/(x_n+1)$ for $1\le i \le n-1$ and $\tilde t = -n^n \, t/((n+1)^{n+1}  x_n \, (x_n+1)^n)$
defines a fibration of the hypersurface~(\ref{mirror_family}) by $f_{n-1}(\tilde x_1, \dots, \tilde x_{n-1} \tilde t)=0$ since
\begin{equation}
\label{fibration_mirrors}
 f_n(x_1,\dots, x_n, t) = x_n \, (x_n+1)^n \, f_{n-1}(\tilde x_1, \dots, \tilde x_{n-1}, \tilde t\,) =0 \;.
\end{equation}
This is the mixed-twist construction with generalized functional invariant $(1,n,1)$.
\end{proof}
\subsection{GKZ data of the mirror family}
\label{ss-GKZ_Dwork}
In the GKZ formalism, the construction of the family $Y^{(n-1)}_{t}$ is described as follows: 
from the homogeneous degrees of the defining Equation~(\ref{eqn-deformed_fermat}) and the coordinates of the ambient projective space 
for the family $X^{(n-1)}_{\lambda}$ we obtain the lattice $\mathbb{L}' = \mathbb{Z}(-(n+1), 1, 1, \dots, 1) \subset \mathbb{Z}^{n+2}$.
We define a matrix $\mathsf{A}' \in \operatorname{Mat}(n+1,n+2;\mathbb{Z})$ as a matrix row equivalent to the $(n+1) \times (n+2)$ matrix with columns of the $(n+1) \times (n+1)$ identity matrix as the first $(n+1)$ columns, followed by the generator of $\mathbb{L}'$:
\begin{equation}
\label{gen_set_prime}
\Scale[0.9]{  
\left( \begin{array}{ccccc} 1 & 0 & 0 & \dots & (n+1) \\ 0 & 1 & 0 & \dots & -1 \\ 0 & \ddots & \ddots & \ddots & -1 \\ \vdots &&&& \vdots \\ 0 & 0 & \dots &0 \qquad 1 & -1
 \end{array}\right)
\quad   \sim \quad \mathsf{A}' = \left( \begin{array}{ccccr} 1 & 1 & 1 & \dots & 1 \\ 0 & 1 & 0 & \dots & -1 \\ 0 & \ddots & \ddots & \ddots & -1 \\ \vdots &&&& \vdots \\ 0 & 0 & \dots & 0 \qquad1 & -1
 \end{array}\right)} \;,
\end{equation} 
and let $\mathcal{A}' =\lbrace \vec{a}'_1, \dots , \vec{a}'_{n+2} \rbrace$ denote the columns of the right-handed matrix obtained by a basis transformation
in $\mathbb{Z}^{n+1}$ from the matrix on the left hand side. 
The finite subset $\mathcal{A}' \subset \mathbb{Z}^{n+1}$ generates $\mathbb{Z}^{n+1}$ as an abelian group and can be equipped
with a group homomorphism $h': \mathbb{Z}^{n+1} \to \mathbb{Z}$, in this case the projection onto the first coordinate, such that $h'(\mathcal{A}')=1$.
This means that $\mathcal{A}'$ lies in an affine hyperplane in $\mathbb{Z}^{n+1}$. The lattice of linear relations between the vectors in $\mathcal{A}'$
is easily checked to be precisely $\mathbb{L}' = \mathbb{Z}(-(n+1), 1, 1, \dots, 1) \subset \mathbb{Z}^{n+2}$.
From $\mathsf{A}'$ we form the Laurent polynomial 
\begin{equation*}
\begin{split}
 P_{\mathsf{A}'}(z_1, \dots, z_{n+1})  & = \sum_{\vec{a}' \in \mathcal{A}'} c_{\vec{a}} \, z_1^{a_1} \cdot z_2^{a_2}\cdots  z_{n+1}^{a_{n+1}}  \\
 &  = c_1 \, z_1 + c_2 \, z_1 \, z_2 + c_3 \, z_1 \, z_3 + \dots + c_{n+2} z_1 \, z_2^{-1} \cdots z_{n+1}^{-1} \;,
\end{split}  
\end{equation*}  
and observe that the dehomogenized Laurent polynomial yields
\begin{multline*}
 \frac{x_1 \cdots x_n}{c_1} \; P_{\mathsf{A}'}\left(1, \frac{c_1 x_1}{c_2} , \frac{c_1 x_2}{c_3}, \dots, \frac{c_1 x_n}{c_{n+1}}\right) \\
= f_n\left(x_1,\dots, x_n, t= (-1)^{n+1} \frac{(n+1)^{n+1} \, c_2 \cdots c_{n+2}}{c_1^{n+1}}\right) \;.
\end{multline*}
\par In the context of toric geometry, this is interpreted as follows:  a secondary fan is constructed from the data $(\mathcal{A}',\mathbb{L}')$. This secondary fan is a complete fan of strongly convex polyhedral cones in $\mathbb{L}^{\prime \vee}_{\mathbb{R}}=\operatorname{Hom}(\mathbb{L}', \mathbb{R})$ which are generated by vectors in the lattice $\mathbb{L}^{\prime \vee}_{\mathbb{Z}}=\operatorname{Hom}(\mathbb{L}', \mathbb{Z})$. As the coefficients $c_1, \dots, c_{n+2}$ -- or effectively $t$ -- vary, the zero locus of $P_{\mathcal{A}'}$ sweeps out the family of hypersurfaces $Y^{(n-1)}_{t}$ in  $(\mathbb{C}^{*})^{n+1}/\mathbb{C}^{*}=(\mathbb{C}^{*})^{n}$. Both $(\mathbb{C}^{*})^{n}$ and the hypersurfaces can then be compactified.  The members of the family  $Y^{(n-1)}_{t}$ are Calabi-Yau varieties since the original Calabi-Yau varieties had codimension one in the ambient space; see Batyrev and van Straten \cite{MR1328251}.
\subsection{Recurrence relation between holomorphic periods}
\label{ss-recurrence}
We now describe the construction of the period integrals.  A result of Doran and Malmendier -- referenced below as Lemma \ref{PeriodLemma} --  shows that the fibration on $Y^{(n-1)}_{t} \to \mathbb{P}^1$ by Calabi-Yau hypersurfaces $Y^{(n-2)}_{\tilde{t}}$ allows for a recursive construction of the period integrals for $Y^{(n-1)}_{t}$ by integrating a twisted period integral over a transcendental homology cycle. It turns out that the result can be obtained explicitly as the Hadamard product of certain generalized hypergeometric functions. Recall that the Hadamard of two analytic functions $f(t) = \sum_{k\geq 0} f_kt^k$, $g(t) = \sum_{k \geq 0} g_kt^k$ is the analytic function $f \star g$ given by 
\begin{equation*}
(f \star g)(t) = \sum_{k=0}^\infty f_kg_kt^k.
\end{equation*}
The unique holomorphic $(n-1)$-form on $Y^{(n-1)}_{t}$ is given by
\begin{equation}
 \eta_t^{(n-1)}= \dfrac{dx_2 \wedge dx_3 \wedge \dots \wedge dx_n}{\partial_{x_1} f_n( x_1,\dots, x_n, t)} \;.
\end{equation}
The formula is obtained from the Griffiths-Dwork technique (see, for example, Morrison \cite{MR1191426}). One then defines an $(n-1)$-cycle $\Sigma_{n-1}$ on $Y^{(n-1)}_{t}$
by requiring that the period integral of $\eta^{(n-1)}_t$ over $\Sigma_{n-1}$ corresponds by a residue computation in $x_1$ to the integral over the middle dimensional 
torus cycle $T_{n-1}(\vec{\mathbf{r}}) := S^1_{r_1} \times \dots \times S^1_{r_{n-1}} \in H_{n-1}(Y^{n-1}_t,  \mathbb{Q})$ with $S^1_{r_j} ~=\lbrace |x|= r_j \rbrace \subset \mathbb{C}$ and $\vec{\mathbf{r}}_{n-1} = (r_1,\dots, r_{n-1}) \in \mathbb{R}^{n-1}_+$, i.e., 
\begin{multline}
\label{residue_integral}
 \underbrace{\int \hspace*{-0.2cm}\dots \hspace*{-0.2cm}\int}_{\Sigma_{n-1}} 
 \dfrac{dx_2 \wedge \dots\wedge  dx_n}{\partial_{x_1} f_n( x_1,\dots, x_n, t)} \\
  = \frac{c_1}{2\pi i} 
   \underbrace{\int \hspace*{-0.2cm}\dots \hspace*{-0.2cm}\int}_{T_{n-1}(r)} P_{\mathcal{A}}\left(1, \frac{c_1 x_1}{c_2} , \frac{c_1 x_2}{c_3}, \dots, \frac{c_1 x_n}{c_{n+1}}\right)^{-1}
 \frac{dx_2}{x_2} \wedge  \dots \wedge   \frac{dx_n}{x_n}.
\end{multline}
The right hand side of Equation~(\ref{residue_integral}) is a resonant $\mathcal{A}$-hypergeometric integral in the sense of \cite[Thm.~2.7]{MR1080980}
derived from the data $(\mathcal{A}',\mathbb{L}')$ and
\begin{equation}
 \vec{\alpha}' =\langle \alpha'_1, -\beta'_1-1, \dots, -\beta'_n-1 \rangle^t =  \langle -1, 0, \dots, 0\rangle^t = \sum_{i=1}^{n+2} \gamma'_i \, \vec{a}'_i
\end{equation}
with $\pmb{\gamma}'_0 = (\gamma'_1, \dots, \gamma'_{n+2}) = ( -1, 0, \dots, 0 )$. We will denote the period integral by $\omega_{n-1}(t) = \oint_{\;\Sigma_{n-1}}  \eta_t^{(n-1)}$. 
\par We recall the following result, which connects the GKZ data above to the iterative twist construction of Doran and Malmendier:
\begin{proposition}
\label{PeriodLemma} \cite[Prop. 7.2]{MR4069107}
For $n \ge 1$ and $|t|\le 1$, there is a family of transcendental $(n-1)$-cycles $\Sigma_{n-1}$ on $Y^{(n-1)}_{t}$
such that
\begin{equation}
   \omega_{n-1}(t)  =  \oint_{\Sigma_{n-1}}  \eta_t^{(n-1)}= (2\pi i)^{n-1} \, \hpg{n}{n-1}{  \frac{1}{n+1} \quad  \dots \quad \frac{n}{n+1} }{ 1 \; \dots\; 1}{ t}  \;.
 \end{equation}   
The iterative structure in Proposition~\ref{MirrorRecursive} induces the iterative period relation
\begin{equation}
\label{iterative_period}
\begin{split}
   \omega_{n-1}(t) & =  (2\pi i) \, \hpg{n}{n-1}{  \frac{1}{n+1} \quad  \dots \quad \frac{n}{n+1} }{ \frac{1}{n} \; \dots\;  \frac{n-1}{n}}{ t}    \star  \omega_{n-2}(t) \quad
    \text{for $n \ge 2$}.
\end{split}   
\end{equation}
Here, the symbol $\star$ denotes the Hadamard product. The cycles $\Sigma_{n-1}$ are determined by $\tilde{T}_{n-1}(\vec{\mathbf{r}}_{n-1}) := \frac{n}{n+1} \cdot \left( T_{n-2}(\vec{\mathbf{r}}_{n-2}) \times S^1_{r_{n-2}}\right)$ as in ~(\ref{residue_integral}), with $r_j ~= ~1 - \frac{j}{j+1}$, and $\frac{n}{n+1}~\cdot \left( T_{n-2}(\vec{\mathbf{r}}_{n-2}) \times S^1_{r_{n-1}}\right)$ indicates that coordinates are scaled by a factor of $\frac{n}{n+1}$.
\end{proposition}
\par  Hence, the iterative structure in Proposition~\ref{MirrorRecursive}, namely, the generalized functional invariant $(1,n,1)$, determines the iterative period relations of the mirror family and the corresponding $\mathcal{A}$-hypergeometric data $(\mathcal{A}',\mathbb{L}',\boldsymbol \gamma_0')$ in the GKZ formalism.
\subsubsection{The mirror family of K3 surfaces}
Narumiya and Shiga \cite{MR1877764} showed that the mirror family of K3 surfaces in Equation~\eqref{mirror_family} with $n=3$ is birationally equivalent to a family of Weierstrass model. In fact, if we set
\begin{equation}
\label{eqn:WEQcoeffs}
\begin{split}
 x_1 & = -{\frac { \left( 4\,{u}^{2}{\lambda}^{2}+3\,X{\lambda}^{2}+{u}
^{3}+u \right)  \left( 4\,{u}^{2}{\lambda}^{2}+3\,X{\lambda}^{2}+{u}^{
3}-2\,u \right) }{{6\lambda}^{2}u \left( 16\,{u}^{3}{\lambda}^{2}-3\,iY
{\lambda}^{2}+12\,Xu{\lambda}^{2}+4\,{u}^{4}+4\,{u}^{2} \right) }} \,,\\
x_2 & = -\,{\frac {16\,{u}^{3}{\lambda}^{2}-3\,iY{\lambda}^{2}+12\,Xu{
\lambda}^{2}+4\,{u}^{4}+4\,{u}^{2}}{8 u \left( 4\,{u}^{2}{\lambda}^{2}+3
\,X{\lambda}^{2}+{u}^{3}-2\,u \right) }} \,,\\
x_3 & = {\frac {{u}^{2}
 \left( 4\,{u}^{2}{\lambda}^{2}+3\,X{\lambda}^{2}+{u}^{3}-2\,u
 \right) }{{2\lambda}^{2} \left( 16\,{u}^{3}{\lambda}^{2}-3\,iY{\lambda
}^{2}+12\,Xu{\lambda}^{2}+4\,{u}^{4}+4\,{u}^{2} \right) }} \,,
\end{split}
\end{equation}
in Equation~\eqref{mirror_family}, we obtain the Weierstrass equation
\begin{equation}
\label{eqn:WEQ}
 Y^2 = 4 X^3 -g_2(u) \, X  - g_3(u)  \,,
\end{equation}
with coefficients 
\begin{equation}
\label{G2G3}
\begin{split}
 g_2 & = \frac{4}{3 \, \lambda^4} \,{u}^{2} \,  \left( u^4 + 8 \lambda^2 u^3 +(4\lambda^2-1)(4\lambda^2+1)u^2+ 8 \lambda^2 u + 1 \right) \,,\\
 g_3 & = \frac {4}{27  \, \lambda^6} \,{u}^{3} \, \left( u^2 +4 {\lambda}^{2} u +1 \right)  \left( 2 u^4 + 16 \lambda^2 u^3 + (32\lambda^4-5)u^2 + 16 \lambda^2u+2\right) \,.
\end{split}
\end{equation}
\par For generic parameter $\lambda$, Equation~(\ref{eqn:WEQ}) defines a Jacobian elliptic fibration with the singular fibers $2 I_4^* + 4 I_1$ and the Mordell-Weil group $\mathbb{Z}/2\mathbb{Z} \oplus \langle 1 \rangle$, generated by a two-torsion section and an infinite-order section of height pairing one; see \cite{MR1877764, MR4099481}.  Using the Jacobian elliptic fibration one has the following:
\begin{proposition}[\cite{MR1877764}]
\label{prop:latticeM2}
The family in Equation~(\ref{eqn:WEQ}) is a family of $M_2$-polarized K3 surfaces with $M_2 \cong  H \oplus E_8(-1) \oplus E_8(-1) \oplus \langle -4 \rangle$ such that the image of the period map is birational with
$\mathcal{M}_{M_2}$.
\end{proposition}
\par Proposition~\ref{prop:latticeM2} shows why the family~(\ref{eqn:WEQ}) can be called the mirror family of K3 surfaces. Dolgachev's mirror symmetry for K3 surfaces identifies marked deformations of K3 surfaces with given Picard lattice $N$ with a complexified K\"ahler cone $K(M) = \lbrace x + i y: \, \langle y, y \rangle > 0, \; x, y \in M_\mathbb{R} \rbrace$ for some mirror lattice $M$; see \cite{MR1420220}. In the case of the rank-one lattice $N_k = \langle 2k \rangle$,  one can construct the mirror lattice explicitly by taking a copy of $H$ out of the orthogonal complement $N_k^\perp$ in the K3 lattice $\Lambda_{K3}$. It turns out that the mirror lattice $M_k \cong  H \oplus E_8(-1) \oplus E_8(-1) \oplus \langle -2k \rangle$ is unique if $k$ has no square divisor. In our situation, the general quartic hypersurfaces in Equation~(\ref{eqn-deformed_fermat}) with $n=3$ have a N\'eron-Severi lattice isomorphic to $N_2=\langle 4 \rangle$, and  the mirror family in Equation~\eqref{eqn:WEQ} is polarized by the lattice $M_2$ such that $N_2^\perp \cong H\oplus M_2$.
\par It turns out that the holomorphic solution of the Picard-Fuchs equation governing the family of K3 surfaces in Equation~(\ref{eqn:WEQ}) equals
\begin{equation}
\label{res1}
   \omega_2  =  \left(  \hpg21{\frac{1}{8}, \, \frac{3}{8}}{1}{  \frac{1}{\lambda^4}} \right)^2  = \  \hpg32 {\frac{1}{4}, \frac{1}{2}, \frac{3}{4}}{1, 1}{  t } \,.
\end{equation}
The first equality was proved by Narumiya and Shiga, and the second equality is Clausen's formula, found by Thomas Clausen, expressing the square of a Gaussian hypergeometric series as a generalized hypergeometric series. 
%, 
%
\subsection{Monodromy of the mirror family}
\label{ss-monodromy}
We will now show how the monodromy representations for the mirror families for general $n$ are computed using the iterative period relations. The results of this section are consistent with the original work of Levelt \cite{MR0145108} up to conjugacy.
\par The Picard-Fuchs operators of the periods given in Proposition~\ref{PeriodLemma} are the associated rank $n$-hypergeometric differential operators annihilating~$\hpgo{n}{n-1}$. But yet more is afforded by pursuing the GKZ description of the period integrals.  In fact, the Euler-integral formula for the hypergeometric functions $\hpgo{n}{n-1}$ generates a second set of \emph{non-resonant} GKZ data  $(\mathcal{A}, \mathbb{L}, \pmb{\gamma}_0)$ from the \emph{resonant} GKZ data $(\mathcal{A}', \mathbb{L}', \pmb{\gamma}'_0)$  by integration. The GKZ data $(\mathcal{A}, \mathbb{L}, \pmb{\gamma}_0)$ determines local Frobenius bases of solutions around  $t=0$ and $t=\infty$. Their Mellin-Barnes integral representation determines the transition matrix between them by analytic continuation.
\par We will always assume that we have $n$ rational parameters, namely $\rho_1, \dots, \rho_n \in (0,1) \cap \mathbb{Q}$, and consider the generalized hypergeometric function
\begin{equation*}
 \hpg{n}{n-1}{  \rho_1 \;  \dots \; \rho_n }{ 1 \; \dots\; 1}{ t} \;,
\end{equation*}
which include all periods from Propositions~\ref{PeriodLemma} and~\ref{prop-one-param}. The \emph{Euler-integral formula} then specializes to the identity
\begin{multline}
\label{Euler-Integral}
\left\lbrack \prod_{i=1}^{n-1}\Gamma( \rho_i) \, \Gamma(1-\rho_i)\right\rbrack \, \hpg{n}{n-1}{  \rho_1 \;  \dots \; \rho_n }{ 1 \; \dots\; 1}{ t} \\
= \left\lbrack\prod_{i=1}^{n-1} \int_0^1 \! \frac{dz_i}{z_i^{1-\rho_i} (1-z_i)^{\rho_i}}\right\rbrack (1-t \, z_1 \cdots z_{n-1})^{-\rho_n} .
\end{multline}
The rank-$n$ hypergeometric differential equation satisfied by $\hpgo{n}{n-1}$ is given by
\begin{equation}
\label{hpg_ode}
\Big\lbrack \theta^n - t \, (\theta + \rho_1) \cdots  (\theta + \rho_n) \Big\rbrack \, F(t) = 0
\end{equation}
with $\theta=t \frac{d}{dt}$, and it has the Riemann symbol
\begin{equation}
\label{RiemannSymbol}
 \mathcal{P}\left. \left(\begin{array}{ccc} 0 & 1 & \infty \\ \hline 0 & 0 & \rho_1 \\ 0 & 1 & \rho_2\\ \vdots & \vdots & \vdots\\ 0 & n-2 & \rho_{n-1} \\
 0 & n -1 -\sum_{j=1}^n \rho_j & \rho_n \end{array} \right| t \right) \;.
\end{equation}
In particular, we read from the Riemann symbol that for each $n \geq 1$, the periods from Proposition~\ref{PeriodLemma} have a point of maximally unipotent monodromy at $t=0$.  This is well known to be consistent with basic considerations for mirror symmetry \cite{MR1265317}.    
\par From the Euler-integral~(\ref{Euler-Integral}),  using the GKZ formalism, we immediately read off
the left hand side matrix, and convert to the A-matrix $\mathsf{A} \in \mathrm{Mat}(2n-1,2n;\mathbb{Z})$ given by
\begin{equation}
\label{gen_set}
\Scale[0.9]{  \left( \begin{array}{ccccccc} 
1 & 1 & 0 & 0 & \dots & 0 & 0 \\0 & 0 & 1 & 1 & \dots & 0 & 0  \\ \vdots &&& & \ddots && \vdots \\  0 & 0 & 0 & 0  & \; \ddots & 1 & 1 \\ \hline
0 & 1 & 0 & 0 & \dots & 0 & 1 \\0 & 0 & 0 & 1 & \dots & 0 & 1  \\ \vdots &&& & \ddots && \vdots \\ 0 & 0 & 0 & 0  & \; \ddots & 0 & 1
\end{array}\right)
\quad   \sim \quad \mathsf{A} = 
\left( \begin{array}{cccc|c|cccc|c}
1	  & 0 & \dots   & 0 		& 0		& 1		& 0 & \dots   & 0		& 0 \\
0	  & 1 &            & 0 		& 0		& 0		& 1 &            & 0 		& 0 \\
\vdots & 	& \ddots & \vdots 	& \vdots	& \vdots	&    & \ddots & \vdots		& \vdots \\ 
0	  & 0 & \dots   & 1		& 0 		& 0  		& 0 & \dots   & 1		& 0\\
\hline
0	  & 0 & \dots   & 0		& 1		& 0	  	& 0 & \dots   & 0		& 1\\
\hline
0	  & 0 & \dots   & 0 		& 1		& 1		& 0 & \dots   & 0		& 0 \\
0	  & 0 &            & 0 		& 1		& 0		& 1 &            & 0 		& 0 \\
\vdots & 	& \ddots & \vdots 	& \vdots	& \vdots	&    & \ddots & \vdots		& \vdots \\ 
0	  & 0 & \dots   & 0		& 1 		& 0  		& 0 & \dots   & 1		& 0\\
 \end{array}\right)} \;,
\end{equation} 
using elementary row operations, as in \S \ref{ss-GKZ_Dwork}.  Let $\mathcal{A} =\lbrace \vec{a}_1, \dots , \vec{a}_{2n} \rbrace$ denote the columns of the matrix $\mathsf{A}$.  The entries for the matrix on the left hand side of (\ref{gen_set}) are determined as follows: the first $n$ entries in each column label which of the $n$ terms $(1-z_i)^{\rho_i}$ or $(1-t \, z_1 \cdots z_{n-1})^{-\rho_n}$ in the integrand of the Euler-integral~(\ref{Euler-Integral}) is specified. For each term, two column vectors are needed and the entries in rows $n+1, \dots, 2n-1$ label the exponents of variables $z_i$ appearing. For example, the last two columns determine the term $(1-t \, z_1 \cdots z_{n-1})^{-\rho_n}$. The finite subset $\mathcal{A} \subset \mathbb{Z}^{2n-1}$ generates $\mathbb{Z}^{2n-1}$ as an abelian group and is equipped with a group homomorphism $h: \mathbb{Z}^{2n-1} \to \mathbb{Z}$, in this case the sum of the first $n$ coordinates such that $h(\mathcal{A})=1$. The lattice of linear relations between the vectors in $\mathcal{A}$ is easily checked to be $\mathbb{L} = \mathbb{Z}(1,\dots, 1, -1, \dots, -1) \subset \mathbb{Z}^{2n}$. The toric data $(\mathsf{A},\mathbb{L})$ has an associated GKZ system of differential equations which is equivalent to the differential equation~(\ref{hpg_ode}). Equivalently, the right hand side of Equation~(\ref{Euler-Integral})  is the $\mathcal{A}$-hypergeometric integral in the sense of \cite[Thm.~2.7]{MR1080980}  derived from the data $(\mathcal{A},\mathbb{L})$ and the additional vector
\begin{multline*}
 \vec{\alpha} \ = \ \langle \alpha_1, \dots, \alpha_{n-1}, -\beta_1-1, \dots, -\beta_n-1 \rangle^t \\
 = \  \langle -\rho_1, \dots, -\rho_n, -\rho_1, \dots, - \rho_{n-1} \rangle^t \ =\ \sum_{i=1}^{2n} \gamma_i \, \vec{a}_i,
\end{multline*}
where we have set  $\pmb{\gamma}_0 = ( \gamma_1, \dots, \gamma_{2n}) = ( 0, \dots, 0, - \rho_1, \dots, - \rho_n) \subset \mathbb{Z}^{2n}$. We always have the freedom to shift $\pmb{\gamma}_0$ by elements in $\mathbb{L}\otimes \mathbb{R}$ while leaving  $\vec{\alpha}$ and any $\mathcal{A}$-hypergeometric integral unchanged.  Thus we have the following:
 \begin{proposition}
 \label{prop:non-resonant}
 The GKZ data  $(\mathcal{A}, \mathbb{L}, \pmb{\gamma}_0)$ is non-resonant.
 \end{proposition}
 \begin{proof}
 We observe that $\alpha_i, \beta_j \not \in \mathbb{Z}$ for $i=1, \dots, n-1$ and $j=1, \dots, n$ and $\sum_i \alpha_i + \sum_j \beta_j \!\equiv - \rho_n \mod{1}\not \in \mathbb{Z}$.  It was proved in \cite[Ex.~2.17]{MR1080980} that this is equivalent to the non-resonance of the GKZ system.
\end{proof}
\subsubsection{Construction of convergent period integrals} 
\label{sssec:convergent}
In this section, we show how from the toric data of the GKZ system convergent period integrals can be constructed. We are following the standard notation for GKZ systems; see, for example, Beukers \cite{MR3545882}. 
\par Let us define the B-matrix of the lattice relations $\mathbb{L}$ for $\mathcal{A}$ as the matrix containing its integral generating set as the rows. Since the rank of $\mathbb{L}$ is 1, we simply have $\mathsf{B}=(1,\dots, 1, -1, \dots, -1) \in \mathrm{Mat}(1,2n;\mathbb{Z}) \cong \mathrm{Hom}_{\mathbb{Z}}(\mathbb{Z}^{2n},\mathbb{Z})$. Of course, the B-matrix then satisfies $\mathsf{A}\cdot \mathsf{B}^{t} =0$, as this is the defining property of the lattice $\mathbb{L}$. The space $\mathbb{L} \otimes \mathbb{R} \subset \mathbb{R}^{2n}$ is clearly a line, and is parameterized by the tuple $(s,\dots, s, -s, \dots, -s) \in \mathbb{R}^{2n}$ with $s \in \mathbb{R}$.  To be used later in this subsection, the polytope $\Delta_{\mathcal{A}}$ defined as convex  hull of the vectors contained in $\mathcal{A}$ is the primary polytope associated with $\mathcal{A}$.  Also for later, we may also write $\mathsf{B} =\sum b_i \hat{e}_i$ in terms of the standard basis $\{\hat{e}_i\}_{i=1}^{2n} \subset \mathbb{Z}^{2n}$. 
\par We can obtain a short exact sequence 
\begin{equation*}
 0 \longrightarrow \mathbb{L} \longrightarrow \mathbb{Z}^{2n}\longrightarrow \mathbb{Z}^{2n-1} \to 0
 \end{equation*}
by mapping each vector $\pmb{\ell} = \sum l_i \hat{e}_i \in \mathbb{Z}^{2n}$
 to the vector $\sum l_i \, \vec{a}_i\in \mathbb{Z}^{2n-1}$. As the linear relations between vectors in $\mathcal{A}$
are given by the lattice $\mathbb{L}$, this sequence is exact. The corresponding dual short exact sequence (over $\mathbb{R}$) is given by
\begin{equation*}
 0 \longrightarrow \mathbb{R}^{2n-1} \longrightarrow \mathbb{R}^{2n} \overset{\pi}{\longrightarrow} \mathbb{L}^{\vee}_{\mathbb{R}} \cong \mathbb{R}
 \longrightarrow 0,
 \end{equation*}
 with $\pi(u_1,\dots, u_{2n})=u_1+\dots+u_n-u_{n+1}-\dots - u_{2n}$. Restricting $\pi$ to the positive orthant in $\mathbb{R}^{2n}$ and calling it $\hat{\pi}$,  we observe that for each $s\in \mathbb{R}$ the set $\hat{\pi}^{-1}(s)$ is a convex polyhedron. For $s \in \mathbb{L}^{\vee}_{\mathbb{R}}$, there are two  maximal cones $\mathcal{C}_+$ and $\mathcal{C}_-$ in the secondary fan of $\mathcal{A}$ for positive and negative real value $s$, respectively. The lists of vanishing components for the vertex vectors in each $\hat{\pi}^{-1}(s)$ are given by
 \begin{equation*}
 \begin{split}
  T_{\mathcal{C}_+} & = \bigcup_{k=1}^{n} \Big\{ \underbrace{\{1, \dots, \widehat{ k}, \dots, n, n+1, \dots 2n \}}_{=: I_k} \Big\}, \\
  T_{\mathcal{C}_-} & = \bigcup_{k=1}^{n} \Big\{ \underbrace{\{1, \dots, n, n+1,  \widehat{ k+n}, \dots \dots 2n \}}_{=: I_{k+n}} \Big\} .
 \end{split}
 \end{equation*}
 The symbol $\widehat{k} $ indicates that the entry $k$ has been suppressed.
 For each member $I$ of $T_{\mathcal{C}_\pm}$, we define $\pmb{\gamma}^{I} = \pmb{\gamma}_0 - \mu^{I} \mathsf{B}$
 such that $\pmb{\gamma}^{I}_i=0$ for $i \not \in I$. We then have
 \begin{equation*}
  \pmb{\gamma}^{I} = \left \lbrace \begin{array}{lll} \pmb{\gamma}_0 & \text{for} \, I \in T_{\mathcal{C}_+}, & \mu^I=0, \\
  ( -\rho_k, \dots, -\rho_k, \rho_k- \rho_1, \dots, 0, \dots, \rho_k- \rho_n) & \text{for} \, I=I_{n+k} \in T_{\mathcal{C}_-}, & \mu^{I_{n+k}}=\rho_k. \end{array} \right.
 \end{equation*}
Then for $I_k \in T_{\mathcal{C}_\pm}$ we denote the \emph{convergence direction} by
\begin{equation}
\label{eqn-convergence_direction}
 \pmb{\nu}^{I_k} =(\nu_1, \dots, \nu_{2p})=(\delta_i^k)_{i=1}^{2p} \in \mathbb{L}\otimes \mathbb{R},
 \end{equation}
where $\delta^k_i$ is the Kronecker delta, such that $\hat{\pi}(\pmb{\nu}^{I_k})=\pm 1$. 
\par Using the B-matrix, one defines the \emph{zonotope} 
 \begin{equation*}
 \mathsf{Z}_{\mathsf{B}} =  \left. \left\lbrace \frac{1}{4} \sum_{i=1}^{2n} \mu_i \, b_i \right| \mu_i \in (-1,1) \right\rbrace 
 = \left( - \frac{n}{2} , \frac{n}{2} \right) \subset \mathbb{L}^{\vee}_{\mathbb{R}} \cong \mathbb{R} .
 \end{equation*}
The zonotope contains crucial data about the nature and form of the solutions to the GKZ system above. A crucial result of  Beukers~\cite[Cor.~4.2]{MR3545882} can then be phrased as follows:
\begin{proposition}\cite[Cor.~4.2]{MR3545882} 
\label{prop:MB}
Let $\pmb{u}, \pmb{\tau}$ be the vector with $\pmb{u}=(u_1, \dots, u_{2n})$, $u_j = |u_j| \, \exp{(2\pi i \tau_j)}$, and $\pmb{\tau}=(\tau_1, \dots, \tau_{2n})$.  For any $\pmb{u}$ with $\pmb{\tau}$ such that $\sum b_i \tau_i \in  \mathsf{Z}_{\mathcal{B}}$ and any $\pmb{\gamma}$ equivalent to $\pmb{\gamma}_0$ up to elements in $\mathbb{L}\otimes \mathbb{R}$ with $\gamma_{n+i} <\sigma < -\gamma_i$ for all $i=1,\dots, n$, the Mellin-Barnes integral given by
\begin{equation}
\label{eqn:MB}
\begin{split}
 \mathsf{M}_{\pmb{\tau}}(u_1, \dots, u_{2n}) 
 = & \int_{\sigma + i \, \mathbb{R}} \left \lbrack \prod_{i=1}^{2n} \Gamma(-\gamma_i - b_i s) \; u_i^{\gamma_i+ b_i s} \right\rbrack \, ds \;,
\end{split}
\end{equation}
is absolutely convergent and satisfies the GKZ differential system for $(\mathsf{A},\mathbb{L})$.
\end{proposition}
\par A toric variety $\mathcal{V}_{\mathcal{A}}$ can be associated with the secondary fan by gluing together certain affine schemes, one scheme for every maximal cone in the secondary fan.  Details can be found in \cite{MR2306158}. In the situation of the hypergeometric differential equation~(\ref{hpg_ode}), the secondary fan has two maximal cones  $\mathcal{C}_+$ and $\mathcal{C}_-$, and  one can easily see that  the toric variety $\mathcal{V}_{\mathcal{A}}$ is the projective line $\mathcal{V}_{\mathcal{A}}=\mathbb{P}^1$ which  is the the domain of definition for the variable $t$ in Equation~(\ref{Euler-Integral}). Each member in the list for a maximal cone contains $2n-1$ integers and define a subdivision of  the primary polytope $\Delta_{\mathcal{A}}$ by polytopes generated by the subdivision, called  regular triangulations. In our case, these regular triangulations are unimodular, i.e.,
 \begin{equation*}
  \text{for all} \, I_k \in T_{\mathcal{C}_\pm}: \quad  \Big| \det\left( \vec{a}_i\right)_{i \in I_k}\Big|=  \Big| \, b_k \Big| =1 \;.
 \end{equation*}
\par Given $\mathcal{A}$ and its secondary fan, we define a ring $\mathcal{R}_{\mathcal{A}}$ by dividing the free polynomial ring in $2n$ variables by the ideal $\mathcal{I}_{\mathcal{A}}$ generated by the linear relations of  $\mathcal{A}$ and the ideal $\mathcal{I}_{\mathcal{C}_{\pm}}$ generated by the regular triangulations. In our situation, we obtain $\mathcal{R}_{\mathcal{A}}$ from the list of generators given by
 \begin{equation*}
 \pmb{\epsilon} = ( \epsilon_1, \dots, \epsilon_{2n} ) = \epsilon \, ( 1, \dots, 1, -1, \dots, -1 ) \in \mathcal{R}_{\mathcal{A}}
 \end{equation*}
 with relation $\epsilon^n=0$, i.e.,  $\mathcal{R}_{\mathcal{A}} = \mathbb{Z}[\epsilon]/(\epsilon^n)$ is a free $\mathbb{Z}$-module of rank $n$.  Thus, we have the following:
 \begin{corollary}
A solution for the hypergeometric differential equation~(\ref{hpg_ode}) is given by restricting to $u_2=\dots=u_{2n}=1$ and $u_1=(-1)^n t$ in Equation~(\ref{eqn:MB}).
 \end{corollary}
\begin{remark}
In the case of the hypergeometric differential equation~(\ref{hpg_ode}), it follows crucially from Beukers \cite[Prop.~4.6]{MR3545882} that  there is a basis of Mellin-Barnes integrals since the zonotope $\mathsf{Z}_{\mathcal{B}}$ contains $n$  distinct points $\{-\frac{n-1}{2}+k\}_{k=0}^{n-1}$ whose coordinates differ by integers.
\end{remark} 
\subsubsection{A basis of solutions around zero}
Using the toric data, we may now derive a local basis of solutions of the differential equation~(\ref{hpg_ode}) around the point $t=0$ \cite{MR2306158}. For the convergence direction $\pmb{\nu}^{I_{1}}$ in $T_{\mathcal{C}_+}$, the $\Gamma$-series is a series solutions of the GKZ system for $(\mathbb{L},\pmb{\gamma}_0)$ and given by
\begin{equation}
\label{Gamma-series}
 \Phi_{\mathbb{L},\pmb{\gamma}_0}(u_1, \dots, u_{2n})  =  \sum_{\pmb{\ell} \in \mathbb{L}} \frac{u_1^{\gamma_1 + \ell_1} \cdots u_{2n}^{\gamma_{2n} + \ell_{2n}}}
{\Gamma(\gamma_1 + \ell_1 + 1) \cdots \Gamma(\gamma_{2n} + \ell_{2n} + 1)} \;.
\end{equation}
\par We have the following:
\begin{lemma}
For the convergence direction $\pmb{\nu}^{I_{1}}$ in $T_{\mathcal{C}_+}$, the $\Gamma$-series for $(\mathbb{L},\pmb{\gamma}_0)$ equals
\begin{equation}
\label{Gamma-series_eval}
\Phi_{\mathbb{L},\pmb{\gamma}_0}(u_1, \dots, u_{2n}) =  \left\lbrack\prod_{i=1}^n \frac{1}{\Gamma(1-\rho_i) \, u_{n+i}^{\rho_i}}\right\rbrack \; \hpg{n}{n-1}{  \rho_1 \;  \dots \; \rho_n }{ 1 \; \dots\; 1}{ t} 
\end{equation}
for $t= (-1)^n u_1 \cdots u_n/(u_{n+1} \cdots u_{2n})>0$. Moreover, convergence of Equation~(\ref{Gamma-series_eval}) in the convergence direction  $\pmb{\nu}^{I_1} =(\nu_1, \dots, \nu_{2p})$ is guaranteed  for all $u_1, \dots, u_{2n}$ with $|u_i|=t^{\nu_i}$ and $0 \le t<1$.
\end{lemma}
\begin{proof}
We observe that
\begin{equation}
\begin{split}
& \Phi_{\mathbb{L},\pmb{\gamma}_0}(u_1, \dots, u_{2n})   \sum_{k\ge 0}  \frac{u_1^{k} \cdots u_n^{k} \cdot u_{n+1}^{-\rho_1-k} \cdots u_{2n}^{-\rho_n-k} } {(k!)^n\, \Gamma(-\rho_1 -k  + 1) \cdots \Gamma(-\rho_n -k + 1)} \\
=& \left\lbrack\prod_{i=1}^n \frac{1}{\Gamma(1-\rho_i) \, u_{n+i}^{\rho_i}}\right\rbrack  \sum_{k\ge 0} \frac{(\rho_1)_k \cdots (\rho_n)_k}{(k!)^n} \, t^k\;.
\end{split}
\end{equation}
The summation over $\mathbb{L}$ reduces to non-negative integers as the other terms vanish when $1/\Gamma(k + 1)=0$ for $k<0$. Using the identities
\begin{equation}
\label{Gamma-identities}
 (\rho)_k = (-1)^k \frac{\Gamma(1-\rho)}{\Gamma(1-k-\rho)} , \quad \Gamma(z) \, \Gamma(1-z) = \frac{\pi}{\sin{(\pi z)}}\;,
\end{equation}
we obtain Equation~(\ref{Gamma-series_eval}). Equation~(\ref{Gamma-series}) shows that restricting the variables $u_2=\dots=u_{2n}=1$ to a base point, the convergence of the $\Gamma$-series $\Phi_{\mathbb{L},\pmb{\gamma}_0}((-1)^n t, 1 \dots, 1)$  is guaranteed for $|u_1| =t$ with $t$ sufficiently small.  
\end{proof}
\begin{remark}
We obtain the same $\Gamma$-series for all convergence directions $\pmb{\nu}^{I_{r}}$ with $1 \le r \le n$ in $T_{\mathcal{C}_+}$. This is due to the fact that in the Riemann symbol~(\ref{RiemannSymbol}) at $t=0$ the critical exponent $0$ has multiplicity $n$. 
\end{remark}
\par However, from the maximal cone $\mathcal{C}_+$ of the secondary fan of $\mathcal{A}$,  we can still construct a local basis of solutions of the GKZ system around $t=0$ by expanding the twisted power series $\Phi_{\mathbb{L},\pmb{\gamma}_0+\pmb{\epsilon}}(u_1, \dots, u_{2n})$  over $\mathcal{R}_{\mathcal{A}}$; see \cite{MR2306158}. Similarly, a twisted hypergeometric series can be introduced, for example, by defining the following renormalized generating function:
\begin{equation}
\label{generating_solutions}
 f(\epsilon, t) = t^\epsilon  \hpgmod{n}{n-1}{  \rho_1 \;  \dots \; \rho_n }{ 1 \; \dots\; 1}{t}  = \sum_{k \ge 0} \frac{(\rho_1+\epsilon)_k \cdots (\rho_n+\epsilon)_k}{(1+\epsilon)_k^n} \, t^{k+\epsilon}.
\end{equation}
We have the following:
\begin{lemma}
For $|t|<1$, choosing the principal branch of $t^\epsilon = \exp{(\epsilon \ln{t})}$ the twisted power series over $\mathcal{R}_{\mathcal{A}}$ is given by
\begin{equation}
\begin{split}
& \Phi_{\mathbb{L},\pmb{\gamma}_0+\pmb{\epsilon}}(u_1, \dots, u_{2n})  
=   \frac{e^{2\pi i \epsilon}}{\Gamma(1+\epsilon)^n} \left\lbrack\prod_{i=1}^n \frac{1}{\Gamma(1-\rho_i-\epsilon) \, u_{n+i}^{\rho_i}}\right\rbrack \; 
   \; t^\epsilon \, \hpgmod{n}{n-1}{  \rho_1 \;  \dots \; \rho_n }{ 1 \; \dots\; 1}{t}.
\end{split}
\end{equation}
\end{lemma}
\begin{proof}
The proof uses $1/(1+\epsilon)_k^n = O(\epsilon^n)=0$ for $k < 0$, where $(a)_k$ is the Pochammer symbol, because for $k \in \mathbb{Z}$ we have
 \begin{equation*}
  \frac{1}{(1+\epsilon)_k} = \frac{\Gamma(1+\epsilon)}{\Gamma(k+1+\epsilon)} =\left\lbrace
  \begin{array}{lcl}  \epsilon(\epsilon-1) \cdots (\epsilon+k+1)& & \text{if $k<0$,}\\ 1 & & \text{if $k=0$,} \\[-0.5em]
   \dfrac{1}{(1+\epsilon)(2+\epsilon)\cdots(m+\epsilon)} & & \text{if $k>0$.} 
 \end{array}\right. 
 \end{equation*}
\end{proof}
For $r=0, \dots, n-1$, we also introduce the functions
\begin{equation*}
 y_r(t) =\frac{1}{r!} \left.\frac{\partial^r}{\partial \epsilon^r}\right|_{\epsilon=0} \! \hpgmod{n}{n-1}{  \rho_1 \;  \dots \; \rho_n }{ 1 \; \dots\; 1}{t}, \;
 y_0(t)= f(0,t)= \hpg{n}{n-1}{  \rho_1 \;  \dots \; \rho_n }{ 1 \; \dots\; 1}{ t} .
\end{equation*}
We have the following:
\begin{lemma}
\label{lem:relation}
For $|t|<1$, the following identity holds
\begin{equation}
\label{eqn:gen_fct}
 f(\epsilon, t) =  \sum_{m=0}^{n-1} \Big(2\pi i  \epsilon\Big)^m \, f_m(t) = \sum_{m=0}^{n-1} \Big(2\pi i  \epsilon\Big)^m \; \sum_{r=0}^m \frac{1}{r!} \, \left(\frac{\ln{t}}{2\pi i}\right)^r \, \frac{y_{m-r}(t)}{(2 \pi i)^{m-r}} ,
\end{equation}
where $f_m(t) = \frac{1}{(2\pi i)^m m!} \frac{\partial^m}{\partial\epsilon^m}\vert_{\epsilon=0} f(\epsilon, t)$ for $m=0, \dots, n-1$.
\end{lemma}
\par As proved in \cite{MR2306158}, the functions $\{ f_r \}_{r=0}^{n-1}$  form a local basis of solutions around $t=0$, and the functions $y_r(t)$ with $r=0, \dots n-1$ are holomorphic in a neighborhood of $t=0$. The local monodromy group is generated by the cycle $(u_1, \dots, u_{2n}) = (R_1 \exp{( i \varphi)},R_2, \dots, R_{2n})$ based at the point $(R_1, \dots, R_{2n})$ for $\varphi\in [0,2\pi]$. Equivalently, we consider the local monodromy of the hypergeometric differential equation generated by $t= t_0 \exp{(i \varphi)}$ for $0<t_0<1$ and $\varphi\in [0,2\pi]$ (by setting $|u_2|=\dots=|u_{2n}|=1$ and $|u_1|=t$). The monodromy of the functions $\{ f_r \}_{r=0}^{n-1}$ can be read off Equation~(\ref{eqn:gen_fct}) immediately. We have the following:
\begin{proposition}
\label{prop:m_0}
The local monodromy of the basis $\pmb{f}^t=\langle f_{n-1}, \dots, f_0\rangle^t$ of solutions 
to the differential equation~(\ref{hpg_ode}) at $t=0$ is given by
\begin{equation}
\label{monodromy_0}
 \mathsf{m}_0 = \left( \begin{array}{ccccc} 1 & 1 & \frac{1}{2} & \dots & \frac{1}{(n-2)!} \\ 0 & 1 & 1 & \dots &  \frac{1}{(n-3)!} \\
 \vdots & \ddots & \ddots &  & \vdots \\ \vdots &  & \ddots & \ddots & 1 \\ 0 & \dots & \dots & 0 & 1\end{array}\right) \;.
\end{equation}
\end{proposition}
\begin{proof}
Lemma~\ref{lem:relation} proves that 
\begin{equation*}
 f_m(t) = \sum_{r=0}^m \frac{1}{r!} \, \left(\frac{\ln{t}}{2\pi i}\right)^r \, \frac{y_{m-r}(t)}{(2 \pi i)^{m-r}} \,.
\end{equation*}
The functions $y_k(t)$ are invariant for $t= t_0 \exp{(i \varphi)}$ for $0<t_0<1$ and $\varphi \to 2\pi$. The result then follows.
\end{proof}
\begin{corollary}
The monodromy matrix $\mathsf{m}_0$ is maximally unipotent.
\end{corollary}
\subsubsection{A basis of solutions around infinity}
We assume $0 < \rho_1 < \dots < \rho_n <1$. Using the toric data we can derive a local basis of solutions of the differential equation~(\ref{hpg_ode}) around the point $t=\infty$. For the convergence direction $\pmb{\nu}^{I_{n+r}}$ in $T_{\mathcal{C}_-}$, the $\Gamma$-series is a series solutions of the GKZ system for  $(\mathbb{L},\pmb{\gamma}^{I_{n+r}})$ and given by
\begin{equation}
\label{Gamma-series2}
\begin{split}
& \Phi_{\mathbb{L},\pmb{\gamma}^{I_{n+r}}}(u_1, \dots, u_{2n})  =  \sum_{\pmb{\ell} \in \mathbb{L}} \frac{u_1^{\gamma_1 - \mu^{I_{r+n}} + \ell_1} 
\cdots u_{2n}^{\gamma_{2n} + \mu^{I_{r+n}} + \ell_{2n}}}
{\Gamma(\gamma_1 - \mu^{I_{r+n}} + \ell_1 + 1) \cdots \Gamma(\gamma_{2n} + \mu^{I_{r+n}} + \ell_{2n} + 1)} \;.
\end{split}
\end{equation}
We have the following:
\begin{lemma}
For the convergence direction $\pmb{\nu}^{I_{n+r}}$ in $T_{\mathcal{C}_-}$ Equation~(\ref{Gamma-series2}) is a series solution for  $(\mathbb{L},\pmb{\gamma}^{I_{n+r}})$. The following identity holds
\begin{equation}
\label{Gamma-series2_eval}
\begin{split}
&  \Phi_{\mathbb{L},\pmb{\gamma}^{I_{n+r}}}(u_1, \dots, u_{2n})   =   
 \frac{e^{\pi i n \rho_r}}{\Gamma(1-\rho_r)^n} \left\lbrack\prod_{i=1}^n \frac{1}{\Gamma(1+\rho_r-\rho_i) \, u_{n+i}^{\rho_i}}\right\rbrack \\
& \qquad  \times  \, t^{-\rho_r} \, \hpg{n}{n-1}{  \rho_r \qquad \dots \qquad \dots \qquad 
\rho_r }{ 1+\rho_r-\rho_1 \;  \dots \widehat{1} \dots \; 1+\rho_r-\rho_n}{ \frac{1}{t}}
 \end{split}
 \end{equation}
for $t= (-1)^n u_1 \cdots u_n/(u_{n+1} \cdots u_{2n})>0$. The symbol $\widehat{1} $ indicates that the entry $1+\rho_r-\rho_i$ for $i=r$ has been suppressed. In particular, restricting variables $u_1=\dots=\widehat{u_{n+r}}=\dots=u_{2n}=1$ the convergence of the $\Gamma$-series  $\Phi_{\mathbb{L},\pmb{\gamma}^{I_{n+r}}}(1, \dots, (-1)^n/t,\dots, 1)$ is guaranteed for $t>1$. 
\end{lemma}
\begin{proof}
A direct computation shows that the $\Gamma$-series satisfies
\begin{equation*}
\begin{split}
 \Phi_{\mathbb{L},\pmb{\gamma}^{I_{n+r}}}(u_1, \dots, u_{2n})  %=  &
% \sum_{k\le  0}  \frac{u_1^{-\rho_r+k} \cdots u_n^{-\rho_r+k}  \cdot u_{n+1}^{\rho_r-\rho_1-k} \cdots u_{2n}^{\rho_r-\rho_n-k} } {\Gamma(1+k-\rho_r)^n\, \Gamma(\rho_r-\rho_1 -k  + 1) \cdots \Gamma(1-k) \cdots  \Gamma(\rho_r-\rho_n -k + 1)} \\
%=& \sum_{k\ge  0}  \frac{u_1^{-\rho_r-k} \cdots u_n^{-\rho_r-k} \cdot u_{n+1}^{\rho_r-\rho_1+k} \cdots u_{2n}^{\rho_r-\rho_n+k} } {\Gamma(1-k-\rho_r)^n\, \Gamma(\rho_r-\rho_1 +k  + 1) \cdots \Gamma(1+k) \cdots \Gamma(\rho_r-\rho_n +k + 1)} \\
=&  \, \frac{e^{\pi i n \rho_r}}{\Gamma(1-\rho_r)^n} \left\lbrack\prod_{i=1}^n \frac{1}{\Gamma(1+\rho_r-\rho_i) \,u_{n+i}^{\rho_i}}\right\rbrack \left( \frac{u_{n+1} \cdots u_{2n}}{(-1)^n u_{1} \cdots u_{n}}\right)^{\rho_r} \\
\times & \sum_{k\ge  0} \frac{(\rho_r)^n_k}{(1+\rho_r-\rho_1)_k \cdots(1+\rho_r-\rho_n + 1)_k} \left( \frac{u_{n+1} \cdots u_{2n}}{(-1)^n u_{1} \cdots u_{n}}\right)^{k} .
\end{split}
\end{equation*}
The result follows.
\end{proof}
Based on the assumption that $0 < \rho_1 < \dots < \rho_n <1$, we have  the following:
\begin{lemma}
There are $n$ different $\Gamma$-series for the convergence directions $\pmb{\nu}^{I_{n+r}}$ with $1 \le r \le n$ in $T_{\mathcal{C}_-}$.
\end{lemma}
\par The local monodromy group is generated by the cycle based at $(R_1, \dots, R_{2n})$  given by $(u_1, \dots, u_{n+r},\dots, u_{2n}) = (R_1, \dots, R_{n+r} \exp{(- i \varphi)}, \dots, R_{2n})$ for $\varphi\in [0,2\pi]$ Equivalently, we consider the local monodromy generated by $t= t_0 \exp{(i \varphi)}$ for $t_0\gg 1$ and $\varphi\in [0,2\pi]$ (by setting $|u_1|=\dots=|u_{2n}|=1$ and $|u_{n+r}|=1/t$). We have the following:
\begin{proposition}
\label{prop:m_infty}
The local monodromy of the basis $\pmb{F}^t=\langle F_{n}, \dots, F_1\rangle^t$ of solutions 
to the differential equation~(\ref{hpg_ode}) at $t=\infty$  is given by
\begin{equation}
\label{monodromy_infty}
 \mathsf{M}_\infty =  \left( \begin{array}{ccccc} e^{-2\pi i \rho_n} & & \\ & \ddots & \\ &&e^{-2\pi i \rho_1}\end{array}\right) \;. 
 \end{equation}
\end{proposition}
\begin{proof}
From the Riemann symbol~(\ref{RiemannSymbol}), we observe that the functions
\begin{equation}
F_r(t) =  A_r \; t^{-\rho_r}  \, \hpg{n}{n-1}{  \rho_r \qquad \dots \qquad \dots \qquad 
\rho_r }{ 1+\rho_r-\rho_1 \;  \dots \widehat{1} \dots \; 1+\rho_r-\rho_n}{ \frac{1}{t}} 
\end{equation}
for $r=1, \dots, n$ and any non-zero constants $A_r$, form a Frobenius basis of solutions  to the differential equation~(\ref{hpg_ode}) at $t=\infty$. The claim follows.
\end{proof}
\subsubsection{The transition matrix}
The solution~(\ref{generating_solutions}) has an integral representation of Mellin-Barnes type \cite{MR3545882} given by
\begin{equation}
\label{BarnesIntegral0}
f(\epsilon, t) = \frac{t^\epsilon}{2 \pi i} \, \frac{\Gamma(1+\epsilon)^n}{\Gamma(\rho_1+\epsilon) \cdots \Gamma(\rho_n+\epsilon)}
\int_{\sigma + i \mathbb{R}} \!\!\!\!\!\!\! ds  \, \frac{\Gamma(s+\rho_1+\epsilon) \cdots \Gamma(s+\rho_n+\epsilon)}{\Gamma(s+1+\epsilon)^n} \cdot
\frac{\pi\, (-t)^s}{\sin{(\pi s)}}\;,
\end{equation}
where $\sigma \in (-\rho_1,0)$. For $|t|<1$ the contour integral can be closed to the 
right. We have the following:
\begin{lemma}
For $|t|<1$, Equation~(\ref{BarnesIntegral0}) coincides with Equation~(\ref{generating_solutions}).
\end{lemma}
\begin{proof}
 For $|t|<1$ the contour integral can be closed to the 
right, and the $\Gamma$-series in Equation~(\ref{generating_solutions}) is recovered as a sum over the enclosed residua
at $r \in \mathbb{N}_0$ where we have used
\begin{equation*}
 \mathrm{for\;all}\;r \in \mathbb{N}_0: \; \operatorname{Res}_{s=r}\left( \frac{\pi\, (-t)^s}{\sin{(\pi s)}} \right) = t^r .
\end{equation*}
\end{proof}
For $|t|>1$ the contour integral must be closed to the left.  The relation to the local basis of solutions at $t=\infty$ can be explicitly computed:
\begin{proposition}
For $|t|>1$, we obtain for $f(\epsilon, t)$ in Equation~(\ref{BarnesIntegral0}) 
\begin{equation}
\label{transition}
f(\epsilon, t) = \sum_{r=1}^n B_r(\epsilon) \, F_r(t)
\end{equation}
where $F_r(t)$ is given by
\begin{equation}
F_r(t) =  A_r \; t^{-\rho_r} \, \hpg{n}{n-1}{  \rho_r \qquad \dots \qquad \dots \qquad 
\rho_r }{ 1+\rho_r-\rho_1 \;  \dots \widehat{1} \dots \; 1+\rho_r-\rho_n}{ \frac{1}{t}} 
\end{equation}
and
\begin{equation}
\label{TransitionFunctions}
 %CHANGE: sign changed in Gamma function of numerator AND 1-rho_i -> 1-rho_r; fix sign in exponential
 A_r = -  e^{-\pi i \rho_r} \prod_{\substack{i=1\\i\not = r}}^n \frac{\Gamma(\rho_i-\rho_r)}{\Gamma(\rho_i) \, \Gamma(1-\rho_r)} , \;
 B_r(\epsilon) =e^{-\pi i \epsilon} \left\lbrack \prod_{i=1}^n \frac{\Gamma(\rho_i) \, \Gamma(1+\epsilon)}{\Gamma(\rho_i+\epsilon)} \right\rbrack
 \frac{\sin{(\pi \rho_r)}}{\sin{(\pi \rho_r+ \pi \epsilon)}} ,
\end{equation}
such that $B_r(0)=1$ for $r=1,\dots, n$.
\end{proposition}
\begin{proof}
 For $|t|>1$ the contour integral in Equation~(\ref{BarnesIntegral0})  must be closed to the left.
Using $1/(1+\epsilon)_k^n = O(\epsilon^n)=0$ for $k<0$, we observe that the poles are located
at $s=-\epsilon-\rho_i-k$ for $i=1, \dots, n$ and $k\in \mathbb{N}_0$. Using
\begin{equation*}
   \forall r \in \mathbb{N}_0: \;  \operatorname{Res}_{s=-r}\Big( \Gamma(s) \, (-t)^{s} \Big) = \frac{t^{-r}}{r!}.
\end{equation*}
and Equations~(\ref{Gamma-identities}) the result follows.
\end{proof}
\par Equation~(\ref{transition}) allows to compute the transition matrix between the Frobenius basis $\langle f_{n-1}, \dots, f_0\rangle^t$ of solutions for the differential equation~(\ref{hpg_ode}) at $t=0$ with local monodromy given by the matrix~(\ref{monodromy_0}) and the Frobenius basis  $\langle F_{n}, \dots, F_1\rangle^t$ of solutions at $t=\infty$ with local monodromy given by the matrix~(\ref{monodromy_infty}). We obtain:
\begin{corollary}
\label{TransitionMatrix}
The transition matrix $\mathsf{P}$ between the analytic continuations of the bases  $\pmb{f}^t=\langle f_{n-1}, \dots, f_0\rangle^t$ at $t=0$ and  $\pmb{F}^t=\langle F_{n}, \dots, F_1\rangle^t$ at $t=\infty$ is given by
\begin{equation}
\label{eqn:transistionmatrix}
 \left(\begin{array}{c} f_{n-1} \\ \vdots \\ f_1\\  f_0\end{array}\right)
 = \left(\begin{array}{ccc} \frac{B_n^{(n-1)}(0)}{(2\pi i)^{n-1} (n-1)!} & \dots &  \frac{B_1^{(n-1)}(0)}{(2\pi i)^{n-1} (n-1)!} \\
 %CHANGE: brackets removed
 \vdots & \ddots & \vdots \\ \frac{B_n^{'}(0)}{2\pi i} & \dots &  \frac{B_1^{'}(0)}{2\pi i} \\[0.2em]
 1 & \dots & 1 \end{array}\right) \cdot \left(\begin{array}{c} F_{n} \\ \vdots \\ F_2\\  F_1\end{array}\right) 
\end{equation}
with $B_r(\epsilon)$ given in Equation~(\ref{TransitionFunctions}). 
\end{corollary}
\begin{proof}
The transition matrix $\mathsf{P}$ between the analytically continued Frobenius basis of solutions $\pmb{f}^t=\langle f_{n-1}, \dots, f_0\rangle^t$ at $t=0$ and the analytic continuation of the Frobenius basis  $\pmb{F}^t=\langle F_{n}, \dots, F_1\rangle^t$ at $t=\infty$ is obtained by first comparing the expression of $f(\epsilon,t)$ from Equation~(\ref{generating_solutions}) as a linear combination of the solutions $\pmb{F}$ at $t=\infty$ from Equation~(\ref{transition}), and subsequently applying Lemma~\ref{lem:relation} to find the explicit linear relations between $\pmb{f}$ and $\pmb{F}$. By differentiation of the functions $B_r(\epsilon)$ in Equation~(\ref{TransitionFunctions}) and evaluating at $\epsilon=0$, we recover the matrix (\ref{eqn:transistionmatrix}).
\end{proof}
We can now compute the monodromy of the analytic continuation of $\pmb{f}$ around any singular point:
\begin{corollary}
\label{cor-all_transitions}
The monodromy of the analytic continuation of $\pmb{f}$ around $t=0$, $t=\infty$, and $t=1$ is given by $\mathsf{m}_{0}$ in~(\ref{monodromy_0}), $\mathsf{m}_{\infty} = \mathsf{P}\cdot \mathsf{M}_{\infty}\cdot \mathsf{P}^{-1}$ for $\mathsf{M}_{\infty}$ in~(\ref{monodromy_infty}), and $\mathsf{m}_{1} =\mathsf{m}_{\infty} \cdot \mathsf{m}_{0}^{-1}$, respectively.
\end{corollary}
\subsubsection{Monodromy after rescaling}
For $C>0$ the rescaled hypergeometric differential equation satisfied by $\tilde{F}(t) =\hpgo{n}{n-1}(Ct)$ is given by
\begin{equation}
\label{hpg_ode_b}
\Big\lbrack \theta^n - C \, t \, (\theta + \rho_1) \cdots  (\theta + \rho_n) \Big\rbrack \, \tilde{F}(t) = 0 \;.
\end{equation}
For $|t|<1/C$ we introduce $\tilde{f}(\epsilon, t)  = C^{-\epsilon} f(\epsilon,C t)$
such that
\begin{equation}
 \tilde{f}(\epsilon, t) =  \sum_{m=0}^{n-1} \Big(2\pi i  \epsilon\Big)^m \, \tilde{f}_m(t) \quad \text{with} \quad
  \tilde{f}_m(t) =\frac{1}{(2\pi i)^m m!} \left. \frac{\partial^m}{\partial\epsilon^m}\right\vert_{\epsilon=0} \!\!f(\epsilon,Ct)
\end{equation}
for $j=0, \dots, n-1$. The local monodromy around $t=0$ with respect to the Frobenius basis $\langle \tilde{f}_{n-1}, \dots, \tilde{f}_0\rangle^t$ is still given by the matrix $\mathsf{m}_0$ in~(\ref{monodromy_0}).  Similarly, for $|t|>1/C$ we introduce $\tilde{F}_k(t) = F_k(Ct)$ for $k=1, \dots, n$. The local monodromy (around $t=\infty$) with respect to the Frobenius basis  $\langle \tilde{F}_{n}, \dots, \tilde{F}_1\rangle^t$ is given by the matrix $\mathsf{M}_{\infty}$ in~(\ref{monodromy_infty}). We obtain:
\begin{proposition}
\label{TransitionMatrix2}
The transition matrix $\tilde{\mathsf{P}}$ between the analytic continuation of $\pmb{\tilde{f}}$ and $\pmb{\tilde{F}}$ such that $\pmb{\tilde{f}} = \tilde{\mathsf{P}} \cdot \pmb{\tilde{F}}$ is given by
\begin{equation}
\label{eqn:Ptilde}
 \tilde{\mathsf{P}}=\Big(\tilde{\mathsf{P}}_{n-j,n+1-k}\Big)_{j=0,k=1}^{n-1,n} \quad\text{with} \quad
 \tilde{\mathsf{P}}_{n-j,n+1-k} = \frac{1}{(2 \pi i)^j j!} \left. \frac{\partial^j}{\partial \epsilon^j} \right\vert_{\epsilon=0} \!\! \Big\lbrack C^{-\epsilon} B_k(\epsilon) \Big\rbrack\;.
\end{equation}
The monodromy of the analytic continuation of $\pmb{\tilde{f}}$ around $t=\infty$ and $t=1/C$ is given by $\mathsf{m}_{\infty} = \tilde{\mathsf{P}}\cdot \mathsf{M}_{\infty}\cdot \tilde{\mathsf{P}}^{-1}$ and $\mathsf{m}_{1/C} =\mathsf{m}_{\infty} \cdot \mathsf{m}_{0}^{-1}$, respectively.
\end{proposition}
\begin{proof}
One emulates the proof of Corollaries \ref{TransitionMatrix} and \ref{cor-all_transitions} directly with new analytic continuations $\pmb{\tilde{f}}$ and $\pmb{\tilde{F}}$ around $t=0$ and $t=\infty$, respectively. In this case, one finds that the functions $B_r(\epsilon)$ appearing in Equation~(\ref{transition}) acquire a factor of $C^{-\epsilon}$. The result then follows suit as claimed.
\end{proof}
In summary, we obtained the monodromy matrices $\mathsf{m}_{0}$ in~(\ref{monodromy_0}), $\mathsf{m}_{\infty} = \tilde{\mathsf{P}}\cdot \mathsf{M}_{\infty}\cdot \tilde{\mathsf{P}}^{-1}$ for $\mathsf{M}_{\infty}$ in~(\ref{monodromy_infty}) and $\tilde{\mathsf{P}}$ in Equation~(\ref{eqn:Ptilde}), and $\mathsf{m}_{1/C} =\mathsf{m}_{\infty} \cdot \mathsf{m}_{0}^{-1}$ for the hypergeometric differential equation $(\ref{hpg_ode_b})$. Thus, we have the following main result:
\begin{theorem}
\label{MirrorRecursive2}
For the family of hypersurfaces $Y^{(n-1)}_{t}$ in Equation~(\ref{mirror_family}) with $n\ge 2$ the mixed-twist construction defines a non-resonant GKZ system. Then a basis of solutions exists given as absolutely convergent Mellin-Barnes integrals whose monodromy around $t=0, 1/C, \infty$ is, up to conjugation, $\mathsf{m}_{0}, \mathsf{m}_{1/C}, \mathsf{m}_{\infty}$, respectively, for $\rho_k = k/(n+1)$ with $k=1, \dots, n$ and $C=(n+1)^{n+1}$.
\end{theorem}
\begin{proof}
The theorem combines the statements of Propositions~\ref{prop:non-resonant}, \ref{prop:MB}, \ref{prop:m_0}, \ref{prop:m_infty}, \ref{TransitionMatrix2} that were proven above.
\end{proof}
We have the following:
\begin{corollary}
Set $\kappa_4=-200\frac{\zeta(3)}{(2\pi i)^3}$, and $\kappa_5=420\frac{\zeta(3)}{(2\pi i)^3}$. The monodromy matrices of Theorem~\ref{MirrorRecursive2} for $2 \le n \le 5$ are given by Table~\ref{tab:monodromy}.  
\end{corollary}
\begin{proof}
We obtain from the multiplication formula for the $\Gamma$-function, i.e.,
\begin{equation*}
 \prod_{k=0}^{m-1} \Gamma\left(z+ \frac{k}{m}\right) = \left(2 \pi\right)^{\frac{1}{2}(m-1)} m^{\frac{1}{2}-m z} \, \Gamma(m z) ,
\end{equation*}
the identity
\begin{equation*}
 C^{-\epsilon} B_k(\epsilon) 
%=  C^{-\epsilon}  \left\lbrack \prod_{i=1}^n \frac{\Gamma(\rho_i) \, \Gamma(1+\epsilon)}{\Gamma(\rho_i+\epsilon)} \right\rbrack \frac{\sin{(\pi \rho_k)}}{\sin{(\pi \rho_k+ \pi \epsilon)}} e^{-\pi i \epsilon} 
= \frac{\Gamma(1+\epsilon)^{n+1}}{\Gamma\big(1+(n+1)\epsilon\big)} \frac{\sin{(\pi \rho_k)}}{\sin{(\pi \rho_k+ \pi \epsilon)}} e^{-\pi i \epsilon} \;.
\end{equation*}
We then compute the monodromy of the analytic continuation of $\pmb{\tilde{f}}$ around $t=0, 1/C, \infty$ where we have set $\kappa_4=-200\frac{\zeta(3)}{(2\pi i)^3}$ and $\kappa_5=420\frac{\zeta(3)}{(2\pi i)^3}$. We obtain the results listed in Table~\ref{tab:monodromy}.  
\end{proof}
The case $n=4$, reproduces up to conjugation  the monodromy matrices for the quintic threefold case by Candelas et al.~\cite{MR1399815} and \cite{MR2369490}.In particular, our results are consistent with the original work of Levelt \cite{MR0145108} up to conjugacy, for any $n >2$. 
\begin{table}[H]
\scalebox{0.62}{
\begin{tabular}{|c|c|ccc|}
\hline
$n$ & $Y^{(n-1)}_t$ & $\mathsf{m}_{0}$ & $\mathsf{m}_{1/C}$ & $\mathsf{m}_{\infty}$ \\
\hline
&&&&\\[-0.9em]
 %CHANGE: M1 corrected
$2$ & EC		& $ \left(\begin{array}{cc} 1 & 1 \\ 0 & 1 \end{array} \right) $ 
			& $ \left(\begin{array}{rc} 1 & 0 \\ -3 & 1 \end{array} \right) $
			& $ \left(\begin{array}{cc} 1 & 1 \\ -3 & -2 \end{array} \right) $\\
&&&&\\[-0.9em]
$3$ & K3 		& $ \left(\begin{array}{ccc} 1 & 1 & \frac{1}{2} \\ 0 & 1 & 1 \\ 0 & 0 & 1\end{array} \right) $
			& $ \left(\begin{array}{rcr} 0 & 0 & -\frac{1}{4} \\ 0 & 1 & 0 \\ -4 & 0 & 0\end{array} \right) $
			& $ \left(\begin{array}{rrr} 0 & 0 & -\frac{1}{4} \\ 0 & 1 & 1 \\ -4 & -4 & -2\end{array} \right) $\\
&&&&\\[-0.9em]
$4$ & CY3 	& $ \left(\begin{array}{cccc} 1 & 1 & \frac{1}{2} & \frac{1}{6} \\ 0 & 1 & 1 & \frac{1}{2} \\ 0 & 0 & 1 & 1 \\ 0&0&0&1\end{array} \right) $
& $ \left(\begin{array}{rrrr} 1+\kappa_4 & 0 & \frac{5\kappa_4}{12} & \frac{\kappa_4^2}{5} \\ 
				-\frac{25}{12} & 1 & -\frac{125}{144} & -\frac{5\kappa_4}{12} \\ 0 & 0 & 1 & 0 \\ 
				-5&0&-\frac{25}{12}&1-\kappa_4\end{array} \right) $
& $ \left(\begin{array}{rrrr} 1+\kappa_4 & 1+\kappa_4 & \frac{1}{2}+\frac{11\kappa_4}{12} & \frac{1}{6}+\frac{7\kappa_4}{12}+\frac{\kappa_4^2}{5} \\ 
-\frac{25}{12} & -\frac{13}{12} & -\frac{131}{144} & -\frac{103}{144}-\frac{5\kappa_4}{12} \\ 0 & 0 & 1 & 1 \\ -5&-5&-\frac{55}{12}&-\frac{23}{12}
-\kappa_4\end{array} \right) $\\
&&&&\\[-0.9em]
$5$ & CY4 	& $ \left(\begin{array}{ccccc} 1 & 1 & \frac{1}{2} & \frac{1}{6} & \frac{1}{24}\\ 0 & 1 & 1 & \frac{1}{2} & \frac{1}{6}
\\ 0 & 0 & 1 & 1 & \frac{1}{2} \\ 0&0&0&1&1 \\ 0&0&0&0&1\end{array} \right)$
& $ \left(\begin{array}{rrrrr} \frac{75}{64} & 0 & \frac{55}{512} & -\frac{11\kappa_5}{384} & -\frac{121}{24576}\\ -\kappa_5 & 1 & -\frac{5\kappa_5}{8} & \frac{\kappa_5^2}{6} & \frac{11\kappa_5}{384}\\ -\frac{15}{4} & 0 & -\frac{43}{32} & \frac{5\kappa_5}{8} & \frac{55}{512} \\ 0&0&0&1&0 \\ -6&0&-\frac{15}{4}&\kappa_5&\frac{75}{64}\end{array} \right)$
& $ \left(\begin{array}{rrrrr} \frac{75}{64} & \frac{75}{64} & \frac{355}{512} & -\frac{11\kappa_5}{384}+\frac{155}{512} &  -\frac{11\kappa_5}{384}+\frac{2399}{24576}\\
-\kappa_5 & -\kappa_5 +1 & -\frac{9\kappa_5}{8}+1 & \frac{(4\kappa_5-3)(\kappa_5-4)}{24} & \frac{\kappa_5^2}{6}-\frac{125 \kappa_5}{384} +\frac{1}{6}\\
-\frac{15}{4} & -\frac{15}{4} & -\frac{103}{32} & \frac{5\kappa_5}{8}-\frac{63}{32} & \frac{5\kappa_5}{8}- \frac{369}{512}\\ 0&0&0&1&1\\
-6 & -6 & - \frac{27}{4} & \kappa_5-\frac{19}{4} & \kappa_5-\frac{61}{64} \end{array} \right)$\\
&&&&\\
\hline
\end{tabular}
\caption{Monodromy matrices for the mirror families with $2 \le n \le 5$}\label{tab:monodromy}}
\end{table}

\bibliography{bibliography}

\end{document}